\documentclass[english,10pt]{amsart}
\usepackage{amsmath, amsfonts, stmaryrd,amssymb,mathtools,esint,mathrsfs}
\usepackage{amsthm,stackrel}

\numberwithin{equation}{section}
\usepackage[french,english]{babel}
 \usepackage[utf8]{inputenc}
 \usepackage[T1]{fontenc}
\usepackage{soul,cancel}

\usepackage[todonotes={textsize=scriptsize}]{changes}
\usepackage[colorlinks=true,linkcolor=blue, citecolor=red]{hyperref}
\usepackage{color}

\usepackage[backend=bibtex]{biblatex}
\usepackage{csquotes}

\newtheorem{thm}{Theorem}
\newtheorem{prop}{Proposition}[section]
\newtheorem{defi}[prop]{Definition}
\newtheorem{lem}[prop]{Lemma}
\newtheorem{coro}[prop]{Corollary}

\newtheorem{rem}[prop]{Remark}

\def\R{\mathbf{R}}

\def\N{\mathbf{N}}
\def\E{\mathbf{E}}

\def\Z{\mathbf{Z}}
\def\T{\mathbf{T}}

\def\R{\mathbb{R}}

\def\N{\mathbb{N}}
\def\E{\mathbb{E}}

\def\Z{\mathbb{Z}}
\def\T{\mathbb{T}}

\def\dd{\mathrm{d}}

\def\B{\mathrm{B}}

\def\W{\mathrm{W}}

\def\C{\mathrm{C}}
\def\L{\mathrm{L}}

\def\H{\mathrm{H}}

\def\ep{\varepsilon}
\def\ffi{\varphi}

\newcommand*{\transp}[2][-2mu]{\ensuremath{\mskip1mu\prescript{\smash{\mathrm t\mkern#1}}{}{\mathstrut#2}}}%

\newcommand{\vertiii}[1]{{\vert\kern-0.08em\vert\kern-0.08em\vert #1 \vert\kern-0.08em\vert\kern-0.08em\vert}}

\title{The Cauchy problem for quasi-linear parabolic systems revisited}
\date{}

\begin{document}
\selectlanguage{english}
\author{{\sc Isabelle Gallagher}}
 \address{DMA, \'Ecole normale sup\'erieure, CNRS, PSL   University, 75005 Paris, France
  and UFR de math\'ematiques, Universit\'e Paris Cit\'e,  75013 Paris, France.
\newline
E-mail: {\tt isabelle.gallagher@ens.fr}}
\author{{\sc Ayman Moussa}}
\address{DMA, \'Ecole normale sup\'erieure, CNRS, PSL University and LJLL, Sorbonne Université, Université Paris Cité, 75005 Paris, France \newline
  E-mail: {\tt ayman.moussa@sorbonne-universite.fr}}
\begin{abstract} 
 We study a class of parabolic quasilinear systems, in which
the diffusion matrix is not uniformly elliptic, but satisfies the
Petrovskii condition (positivity of eigenvalues' real part).
Local wellposedness is known since the work of Amann in the 90s, by a
semi-group method.  We first revisit these results in the context of Sobolev spaces modelled on $\textnormal{L}^2$ {and then explore the endpoint Besov case $B^{d/p}_{p,1}$}. We also exemplify our method on the SKT system, showing the existence of local, non-negative, strong solutions.    \end{abstract}

\maketitle

\section{Introduction}
\subsection{Main results}
This article deals with local well-posedness for the following quasilinear parabolic system, set on the~$d$-dimensional torus~$\T^d$:
\begin{align}
\label{eq:PtNL}
  \left\{
  \begin{aligned}
    \partial_t U-\sum_{k=1}^d\partial_k \big[A(U)\partial_k U\big] &= F\, ,\\
   U_{|t=0} &= U^0.
  \end{aligned}
  \right.                                
\end{align}
In this system $U^0: \T^d \rightarrow \R^N$ and $F:\R_{\geq 0}\times \T^d\rightarrow \R^N $ are given, $A:\R^N\rightarrow\textnormal{M}_N(\R)$ is a smooth matrix field and $U:\R_{\geq 0}\times \T^d\rightarrow \R^N $ is the unknown.  Our analysis will rely on a detailed study of the linear case in which a matrix field $M$ is given and one searches for~$V$ solving
\begin{equation}
\label{eq:Pt}
  \left\{
  \begin{aligned}
    \partial_t V-\sum_{k=1}^d\partial_k \big[M\partial_k V\big] &= F\, ,\\
   V_{|t=0} &= V^0.
  \end{aligned}
  \right.                                
\end{equation}
As we shall  see later on, this system already hides several difficulties in order to build a well-posedness theory with propagation of Sobolev norms. As a matter of fact, the following spectral condition will be of utmost importance in our construction (we refer to Section~\ref{appendix:petrovskii} for more on that condition).
\begin{defi}[Petrovskii condition]
A matrix~$B\in\textnormal{M}_N(\R)$ satisfies the Petrovskii condition if it belongs to~$\mathscr{P}$, where
 \begin{equation}   \label{eq:Pet}\mathscr{P} :=\bigcup_{\delta>0}\mathscr{P}_\delta \, ,\quad \mbox{with} \quad \forall \delta \in \R \, , \quad   \mathscr{P}_\delta  := \Big\{B\in\textnormal{M}_N(\R)\,:\,z\in \textnormal{Sp}(B) \Rightarrow \textnormal{Re}(z)\geq \delta\Big\}\, .
\end{equation}
\end{defi}
With this Petrovskii condition at hand, we are in position to state our two main results. We anticipate a notation that will be introduced in Paragraph~\ref{subsec:not} : for $s\in\R$ and~$T>0$ we note $E_T^s$ for the $\H^s(\T^d)$-energy space $\mathscr{C}^0([0,T];\H^s(\T^d))\cap\L^2(0,T;\H^{s+1}(\T^d))$ that we equip  with the norm
   \[\|U\|_{E_T^s}  :=\Big(\|U\|_{\mathscr{C}^0([0,T];\H^s(\T^d))}^2 + \int_0^T \|\nabla U(t)\|_{\H^s(\T^d)}^2\,\dd t\Big)^{1/2}.\] 
   We also define~$X_T^s:=\mathscr{C}^0([0,T];\H^s(\T^d))$ and~$Y_T^s:=\L^2(0,T;\H^{s}(\T^d))$, as well as the set~$Q_T:=[0,T] \times \T^d$.

 The initial data~$U^0 $ will be chosen in a Sobolev space~$ \H^s(\T^d)$, while the force~$F$  will correspondingly lie  in~$ Y_T^{s-1}$ for some~$T>0$. We denote the  functional framework for the data  by~$
   \mathcal D_T^s:= \H^s(\T^d) \times Y_T^{s-1}
      $
      and the size of the data is  measured by
         $$
\|(U^0,F)\|_{  \mathcal D_T^s}:= \|U^0\|_{\H^s(\T^d)} + \|F\|_{Y_T^{s-1}} + \int_0^T \left|\langle F(t)\rangle\right|\,\dd t.
$$
We have denoted by~$ \langle f\rangle$ the average of any function~$f$ over~$\T^d$.
Our main result is the following.
    \begin{thm}[Local well-posedness] \label{thm:exloc}
   Consider a smooth~$A:\R^N\rightarrow \mathscr{P}$, and~$s>d/2$. For any~$(U^0,F)$ belonging to~$\mathcal D_\infty^s$ there exists  $T>0$ and a unique element $U$ of $E_T^s$ which solves the parabolic Cauchy problem  \eqref{eq:PtNL} on $Q_T$.  Moreover, if $(U^0_1,F_1),(U^0_2,F_2)$ lie  in the ball of size~$R$ in~$\mathcal D_\infty^s$, with~$U_1$ and $U_2$ the respective solutions both belonging to the ball of size~$R'$ of~$E_T^s$ for some~$T>0$, then there is a constant~$C $ depending on~$R$ and~$R'$  such that        \begin{align}\label{eq:stabNL}
\|U_1-U_2\|_{E_T^s} \leq C \big \|(U^0_1-U^0_2,F_1-F_2)\big\|_{  \mathcal D_T^s}.
\end{align}
\end{thm}
The restriction~$s>d/2$ guarantees some H\"older continuity in space (of  order~$s-d/2$) of the initial data and the solution. It is to be compared with the restriction~$1>d/p$ of the related work of Amann~\cite{amann} in~$W^{1,p}$ which we present  in Paragraph~\ref{state of the art} below.   {It is actually possible to reduce to a "scaling zero" result by resorting to Besov spaces: a theorem in the framework of~$B^\frac dp_{p,1}$, for any finite~$p$, is presented in Section~\ref{end-point Besov}.}

With this (local in time) well-posedness setting we can define the lifetime of the solution associated with~$(U^0,F)\in \mathcal D_\infty^s$ for some~$s>d/2$  by
\begin{align*}
T^\star_s(U^0,F) := \sup \big\{T>0\,:\, \exists\, U \in E_T^s\textnormal{ solving \eqref{eq:PtNL} on } Q_T\big\}.
\end{align*}
The construction leading to Theorem~\ref{thm:exloc} provides  a lifetime $T^\star_s$ which depends   on the data~$(U^0,F)$ not only through its size but also (in some sense) through its form. Actually
it is possible, thanks to a propagation of regularity result, to prove that the lifetime actually only depends on the size of the data.
Also, if the data is small enough the lifetime is infinite. We have more precisely the following result.  \begin{thm}[Lifetime and blow-up] \label{thm:lbu}
   Consider the assumptions of Theorem~{\rm\ref{thm:exloc}}.
   \begin{itemize}
   \item[$(i)$] There exists $\ep>0$ depending only on $A$ and $s$ such that
     \begin{align*}
\|(U^0,F)\|_{  \mathcal D_\infty^s} < \ep \Longrightarrow T_s^\star(U^0,F) = \infty\, .
     \end{align*}
   \item[$(ii)$]      
There exists a decreasing function $\ffi$   such that  $T^\star_s(U^0,F) \geq \ffi(\|(U^0,F)\|_{  \mathcal D_\infty^s})$.
   \item[$(iii)$] If $T^\star_s(U^0,F)<+\infty$, then $\lim_{t\rightarrow T^\star_s(U^0,F)}\|U(t)\|_{\H^s(\T^d)} = \infty$.
   \end{itemize}
 \end{thm}
 {Of course, in the previous results, lower order nonlinearities may be added without changing the conclusion of these statements except point $(i)$ of the second one (unless structural assumptions on the lower order term are added). For instance, estimate \eqref{eq:stabNL} of Theorem~\ref{thm:exloc} allows for a direct use of Picard's fixed-point theorem to establish the following corollary, useful for cross-diffusion systems.}

 \begin{coro}\label{coro:reac}
   Consider a smooth function $R:\R^N\rightarrow\R^N$ and the assumptions of Theorem~{\rm\ref{thm:exloc}}. All conclusions of Theorem~{\rm\ref{thm:exloc}} and Theorem~{\rm\ref{thm:lbu}} except point $(i)$ of the latter hold for the following system
   \begin{align}
\label{eq:PtNL:reac}
  \left\{
  \begin{aligned}
    \partial_t U-\sum_{k=1}^d\partial_k \big[A(U)\partial_k U\big] &= F+R(U)\, ,\\
   U_{|t=0} &= U^0.
  \end{aligned}
  \right.                                
\end{align}
 \end{coro}

\subsection{State of the art}\label{state of the art}
\ 
 {In this article, we focus on parabolic systems, which must be distinguished from their scalar counterparts (parabolic equations). For this reason, the literature reviewed in this section concerns only systems. The difficulties encountered in the study of scalar parabolic equations are of a different nature; this is partly due to the maximum principle which plays an important role in the scalar case, and is absent in the case of systems. For the scalar case (specifically linear parabolic equations), we refer, for instance, to the classical work of Krylov \cite{krylov}, which is a standard reference on the subject.} Parabolic systems have been studied for a long time. The pioneer contribution of Petrovskii \cite{petrorigine} seems to be the starting point of the story. Back then existence and uniqueness of solution for parabolic \emph{linear} systems was yet to be explored. The major step of Petrovskii in this context is the discovery of a condition on those linear systems ensuring the existence and uniqueness of a global solution. For a parabolic system in divergence form (non-constant coefficients $a_{ij}^{k\ell}$, unknown $U:=(u_i)_{1\leq i\leq M}$ considered on~$\R_{\geq 0}\times\T^d$) 
\begin{align}\label{eq:u}
\partial_t u_i - \sum_{k,\ell=1}^d \sum_{j=1}^M \partial_k \Big[a_{ij}^{k\ell} \partial_\ell u_j \Big] = 0 \, , 
\end{align}
Petrovskii's condition requires that for all vectors~$\xi\in\R^d$ of euclidean norm $1$, the matrix field~$A_\xi := (\sum_{k,\ell} a_{ij}^{k\ell} \xi_k\xi_\ell)_{ij}$ has a spectrum lying in the set~$\{z\in\mathbb{C}\,:\,\textnormal{Re}(z)>0\}$. In the literature some references can be found in which this previous condition is replaced by $\langle a_\xi X,X\rangle \geq 0$ which amounts to asking that the \emph{symmetric part} of the matrix field $A_\xi$ satisfies Petrovskii's condition. Exploring the very same linear system of equations \eqref{eq:u} with this latter assumption is by far more restrictive and actually flushes out all the subtlety of the problem, because under this condition the system \eqref{eq:u} has obvious energy estimates.  {As opposed to Petrovskii's, this stronger ellipticity condition \emph{is not} intrinsic to the system for it depends on set of coordinates chosen for $U$: one can clearly exhibit a symmetric positive matrix $A$ and an invertible matrix $P$ for which $PAP^{-1}$ does not have a positive symmetric part. Let us precise that such a symmetry assumption on the same system \eqref{eq:u} \emph{does not} fall into the scope of any symmetrization procedure (which usually changes the evolution form of the system). We will elaborate on these transformations later on, when dealing with the quasilinear setting.} For the moment, let us proceed with the state of the art for the linear case, focusing only in the literature which treat systems following Petrovskii's condition. In \cite{petrorigine}, the tensor field $A=(a_{ij}^{k\ell})_{i,j,k,\ell}$  depends only on the time variable and the setting is rather regular both for the data and the   solution; existence of a solution is obtained by means of a fundamental solution. Let us precise however that Petrovskii's condition and construction also hold for higher order parabolic systems. In the two decades following \cite{petrorigine}, several contributions extended this study of the Cauchy problem to more general systems and in  less regular settings, see \cite{aronson,lady,eidelman} to cite a few. We also refer to the bibliographical remarks section of the book \cite{friedman} of Friedman and to the monograph \cite{ladyfat},  {especially \cite[Chapter VII, \S~8]{ladyfat} in which Petrovskii's condition is stated for higher order systems together with a comprehensive description of the existing literature at that time}. All those cited works rely on the condition exhibited by Petrovskii, with a construction of the fundamental solution (exception made of \cite{mizo} which relies on semigroup theory). Further generalizations of this condition encompassing even more general systems and sets of functions have been explored, see for instance the review \cite{pala} on the Gel'fand-Shilov theory for parabolic systems. 

\medskip

Leaving the realm of \emph{linear} systems, the literature is by far less generous. For quasilinear parabolic systems, Amann's work \cite{amann} seems to be the only reference covering a variety of cases comparable to the Petrovskii theory for linear systems. Without surprise, the work of Amann relies crucially on Petrovskii's condition (Amann speaks of the \emph{normal ellipticity condition}). Instead of the linear system \eqref{eq:u}, Amann tackles the following quasilinear one
\begin{align}\label{eq:quasi:u}
\partial_t u_i - \sum_{k,\ell=1}^d\sum_{j=1}^M \partial_k \Big[a_{ij}^{k\ell}(U) \partial_\ell u_j \Big] = 0\, , 
\end{align}
where the tensor $A:=(a_{ij}^{k\ell})_{i,j,k,\ell}$ now depends on as many variables as the system and takes its values in the set of tensors satifying Petrovskii's condition.  {An important remark is in order. Quite frequently, diffusive systems arising from physics, chemistry or biology offer the dissipation of some functional (an entropy) along the flow of their solution. It is also known, at least formally, that the existence of such an entropy is equivalent to a symmetrization procedure for the system \cite{degond,kawashizu,kawa}. This type of transformation allows to rewrite the system \eqref{eq:quasi:u} for another set of unknowns (the entropic variables) $V:=(v_i)_i$ solving}\begin{align}\label{eq:u:symm}
\sum_{j=1}^N c_{ij}(V) \partial_t v_j - \sum_{k,\ell=1}^d\sum_{j=1}^M \partial_k \Big[b_{ij}^{k\ell}(V) \partial_\ell v_j \Big] = 0\, .
\end{align}
 {In this new formulation, the matrix field $C:= (c_{ij})_{i,j}$ and the tensor field $B:=(b_{ij}^{k\ell})_{i,j,k,\ell}$ are symmetric---hence the term \emph{symmetrization}. As noticed in \cite{gio}, whenever this procedure can be carried out, Petrovskii's condition is equivalent to the positivity of the symmetric tensor~$B(V)$. However, for the symmetrized system~\eqref{eq:u:symm}, the positivity of $B(V)$ is less useful --- compared to that of $A(U)$ in~\eqref{eq:quasi:u} --- for the construction of strong solutions. Indeed, the symmetrized formulation is particularly well suited for studying entropy dissipation or for applying perturbative methods to construct local (or global) solutions near an equilibrium. On the other hand, it is less appropriate than~\eqref{eq:quasi:u} for deriving energy estimates involving spatial derivatives of $U$, or for using semigroup theory in the analysis of evolution equations. That is probably the main reason why, even though applying his theory to several physical models enjoying a symmetrization property, Amann only focused on the formulation \eqref{eq:quasi:u} (under Petrovskii's condition) to produce local well-posedness theorems leading to Sobolev-valued solutions.} More precisely, it is proven in \cite{amann} that given $p>d$ and any initial data in $\W^{1,p}(\T^d)$ there exists a unique $\W^{1,p}(\T^d)$-valued solution~$U$ to \eqref{eq:quasi:u},  in a vicinity of the origin in $\R_{\geq 0}$; if the maximal lifetime of this solution is finite, then blow-up occurs in the $\W^{1,p}(\T^d)$ norm. As a matter of fact, Amann's theory allows for even more complicated systems: it encompasses the boundary-value problem on a domain of $\R^d$, with extra dependence on $t$ and $x$ for the tensor $A$ and more (lower order) terms in the system. Amann's theory is a highly complex machinery relying on several non-trivial ingredients: general interpolation, maximal regularity and analytic semigroup theory. We also mention that 10 years ago, Pierre-Louis Lions gave a series of lectures on parabolic systems \cite{pilou} in which part of the bibliographical material that we have cited here is presented together with a possible strategy to build local solutions.

Systems like \eqref{eq:quasi:u} (quite often with a non-vanishing source term) arise naturally in several contexts as a model of \emph{diffusion} in a multicomponent setting. The choice of a diagonal diffusion tensor $A$ (that is $A_{ij}^{k\ell} = 0$ for $k\neq \ell$) corresponds to standard or \emph{isotropic} diffusion while non-diagonal diffusion tensor corresponds to \emph{anisotropic} diffusion. The latter case can be preferred when the quantities at stake evolve in a highly heteregeneous environnement in which the brownian motion from which \eqref{eq:quasi:u} originates  is not completely symmetric in all directions. These types of models do exist (see for instance \cite{chen_perthame} or \cite{hillen}) but their use is rather limited in comparison with the isotropic case. For this reason and because this work originates from questions arising in population dynamics (see the SKT model below) we have chosen to focus here only on systems of the form \eqref{eq:PtNL}, that is exactly the case of isotropic diffusion. This class of systems already contains a large number of models, the mathematical analysis of which is highly non-trivial.   {This includes renowned cases of models describing chemical concentrations, cell density, gaz mixtures or population densities. All those systems, originally introduced in a modelling purpose, offered to the mathematical community genuine and challenging questions about their behavior, be it existence and uniqueness of solutions, blow-up or long-time behavior. For instance, the sole case of gaseous mixtures \emph{via} fully coupled nonlinear models -- such as the Maxwell-Stefan system -- has given rise to a substantial literature. We can cite for instance the pioneering works \cite{gio,giomassot}, which established well-posedness near equilibrium or more recently, the works of \cite{bothe} and~\cite{jungstel} who extended this analysis in non-perturbative regimes: one using Amann's theory to construct local solutions, the other relying on entropy dissipation methods to obtain global weak solutions. For a more comprehensive description of those diffusive models and the associated references, we refer to \cite[Chapter 4]{ansgar} or the recent lecture notes~\cite{jung_lecture_notes}.}  

\medskip

A common feature shared by those models is their \emph{cross-diffusion} aspect: even though the diffusion operator used on each component is \emph{isotropic}, several components of the system undergo  the influence of other components on the \emph{intensity} of its diffusion. It is a remarkable fact that even if the environnement in which the components evolve is completely isotropic, the sole mutual influence on the intensity of their diffusion can lead to asymmetric patterns. A spectaculary instance of this phenomenon is observable in the SKT (for Shigesada, Kawaski and Teramoto) model introduced in \cite{skt}. In this cross-diffusion system (which falls into the scope of \eqref{eq:PtNL}), even though the diffusion is isotropic, stable segregation steady states are possible corresponding to cases in which each of the species shares out the available space, in some sense. The SKT model and its generalizations are iconic examples of the possible use of Amann's theory. If global weak solutions are known to exist thanks to the (rather lately discovered) entropy structure for those systems (see \cite{chen_jun,dlmt} and the references therein), as far as our knowledge goes the only way to build (local) strong solutions is to rely on Amann's theory (as noticed by Amann himself in \cite{amann}). From this step, a considerable amount of attempts to prove the existence of \emph{global} strong solutions to the SKT system (or its variants) emerged (see \cite{HNP,GMT} and the references therein for the most recent improvements). In all those works, Amann's theory is used as a black box and the quest is reduced to the denial of the blow-up criterion which holds in case of finite lifetime, as established by Amann.

\medskip

This work aims at proposing an alternative approach to the construction of local strong solutions for quasilinear parabolic system like \eqref{eq:PtNL} (satisying Petrovskii's condition), using relatively few elaborate tools, in comparison with Amann's construction. From this point of view, our approach differs from \cite{amann} by the fact that we do not use any abstract result on parabolic equations (no semi-group theory nor maximal regularity) but we rely instead on Fourier analysis and the paraproduct  of Bony \cite{bony} to treat the most severe non-linearities of the system. In this way, we manage to build solutions in a finer scale of spaces, but yet comparable: our solutions live in $\H^s(\T^d)$ for $s>d/2$ whereas Amann's in $\W^{1,p}(\T^d)$ for~$p>d$. In the specific example of the SKT model, we hope that this new path will shed some light on the question of the possible blow-up of these solutions, at least in the periodic setting that we consider. 

\subsection{Sign-preserving property and application to the SKT model}

As repeatedly noticed by Amann \cite{amann} and contrary to the scalar case, parabolic systems satisfying Petrovskii's condition do not offer any maximum principle. When dealing with diffusive models aiming at describing the evolution of \emph{densities}, the non-negativity of the solution is a crucial property of the model that one would like to propagate from the initial data. This is not a harmless detail from the point of view of mathematical analysis either as it may happen (see below for some examples) that Petrovskii's condition is only satisfied on the cone of non-negative vectors. This transference of non-negativeness (component-wise) from the initial data to the solution on its whole lifetime is tightly linked to the structure of the system. We give below a sufficient condition on the matrix field to ensure this propagation. This condition was originally suggested at the formal level in \cite{pilou} for the general case of parabolic systems that we consider in this article {; note that an analogous condition has also been introduced in the specific case of multicomponent gas models in \cite[Section 7.3.3]{gio}}. In order to motivate the following definition, notice that in the case of systems, the non-negativity of  the diagonal part of the operator  alone does not ensure the preservation of the sign of the solution, due to the  presence of lower-order terms. These terms will   not affect the preservation of sign only if they are themselves in some sense diagonal, as presented in the coming Definition~\ref{def:nonneg}. Proposition~\ref{prop:sign} and Theorem~\ref{thm:SKT} below enlight  the relevance of that definition  in our setting of solutions. 

For $V\in\R^N$ the notation $\textnormal{diag}(V)$ refers to the diagonal square matrix of size $N$ with entries given by the components of $V$. The  partial order $\geq$ on $\R^N$ or $\textnormal{M}_N(\R)$ has to be understood component-wise. 
\begin{defi}\label{def:nonneg}
  A matrix field $A:\R^N\rightarrow\textnormal{M}_N(\R)$ is said to be \textsf{sign-preserving} if there exist smooth maps $D:\R^N\rightarrow\textnormal{diag}(\R^N)$ and $B:\R^N\rightarrow \textnormal{M}_N(\R)$ such that
  \begin{itemize}
  \item[$\bullet$] $A(U)= D(U)+\textnormal{diag}(U)B(U)$ ;
    \item[$\bullet$] for some nonnegative real number $\alpha$ and any $U\geq 0$ one has $D(U)\geq \alpha\textnormal{I}_N$.
  \end{itemize}
\end{defi} \ 
The relevance of that  definition stems from the following proposition, proved in Section~\ref{proofprop:sign}.
\begin{prop}\label{prop:sign}
Fix~$s>d/2+2$. Let $\rho:\R^N\rightarrow\R^N$ be a given smooth function and define~$R(U):=\textnormal{diag}(U)\rho(U)$. Consider a smooth sign-preserving matrix field $A$, and a solution~$U$ to the Cauchy problem \eqref{eq:PtNL:reac}
in~$E_T^s$ associated with
non-negative~$(U^0,F)\in\mathcal{D}_\infty^s$. Then~$U$ is non-negative on~$[0,T]$.
\end{prop}

Finally let us state the following  theorem, which is a consequence  of our main result and Proposition~\ref{prop:sign}, and will be applied to the SKT system below. Its proof can also be found in Section~\ref{proofprop:sign}.
\begin{thm}\label{thm:SKT}
Fix~$s>d/2$. Let $\rho:\R^N\rightarrow\R^N$ be a given smooth function and define~$R(U):=\textnormal{diag}(U)\rho(U)$. Consider a smooth sign-preserving matrix field $A$  satisfying $A(\R_{\geq 0}^N)\subset \mathscr{P}$, and  non-negative $(U^0,F)\in\mathcal{D}_\infty^s$. There exists $T>0$ and a unique element~$U$ of $E_T^s$ which solves the Cauchy problem \eqref{eq:PtNL:reac}. Moreover, we have the stability estimate \eqref{eq:stabNL} and points~$(ii)$ and $(iii)$ of Theorem~{\rm\ref{thm:lbu}};  point~$(i)$ holds if~$\rho$ vanishes identically.  Finally, $U$ is non-negative on its whole lifetime.
\end{thm}
\begin{rem}
As compared to Theorem~{\rm\ref{thm:exloc}} and Theorem~{\rm\ref{thm:lbu}}, Petrovskii's condition is here only required to be satisfied on the cone $\R^N_{\geq 0}$. 
\end{rem}
Our interest in this question originates from the study of the SKT model \cite{skt}. We end this paragraph by an example of use of Theorem~\ref{thm:SKT}
  on this specific system. In its original form, the SKT model writes
\begin{align}
\label{eq:SKT}
  \left\{
  \begin{aligned}
    \partial_t u_1- \Delta[(d_1+a_{11}u_1+a_{12}u_2)u_1] = u_1(r_1-s_{11}u_1-s_{12}u_2)\, ,\\
    \partial_t u_2- \Delta[(d_2+a_{21}u_1+a_{22}u_2)u_2] = u_2(r_2-s_{21}u_1-s_{22}u_2)\,,
  \end{aligned}
  \right.                                
\end{align}
where the unknowns $u_1,u_2:\R_{\geq 0}\times\T^d\rightarrow\R_{\geq 0}$ are density population and all the coefficients~$a_{ij}$, $r_i$, $s_{ij}$ are nonnegative while the $d_i$'s are positive. Such a system can be written in the form \eqref{eq:PtNL:reac} for $U:=\transp{}{(u_1,u_2)}$,
\begin{align*}
  R(U) = \begin{pmatrix}
    u_1 & 0 \\
    0 & u_2
  \end{pmatrix}\begin{pmatrix}
    r_1-s_{11}u_1-s_{12}u_2\\
    r_2-s_{21}u_1-s_{22}u_2
   \end{pmatrix}
\end{align*}
and the matrix field
\begin{align}\label{def:ASKT}
  A_{\textnormal{SKT}}:\begin{pmatrix} u_1\\u_2\end{pmatrix} \longmapsto \begin{pmatrix} d_1+2a_{11}u_1+a_{12}u_2 &  a_{12}u_1\\
    a_{21}u_2 & d_2+ a_{21}u_1+2 a_{22}u_2
  \end{pmatrix}.
\end{align}
Writing \begin{align*}
        A_{\textnormal{SKT}}(U)= \begin{pmatrix} d_1+a_{12}u_2 & 0\\
    0 & d_2+ a_{21}u_1
    \end{pmatrix} +\begin{pmatrix} 2a_{11}u_1 &  a_{12}u_1\\
    a_{21}u_2 & 2 a_{22}u_2
    \end{pmatrix},
\end{align*}
we see that this matrix field is indeed sign-preserving in the sense of Definition~\ref{def:nonneg}. Lastly, we have that $A_{\textnormal{SKT}}(\R_{\geq 0}\times\R_{\geq 0})\subset\mathscr{P}$ : it can be readily checked that $\det A_{\textnormal{SKT}}(u_1,u_2)$ and $\textnormal{Tr}\, A_{\textnormal{SKT}}(u_1,u_2)$ are both positive for $u_1,u_2\geq 0$ (because the $d_i$'s are positive), so either the eigenvalues are not real and share a positive real part, or they are both real and have the same (positive) sign. Theorem~\ref{thm:SKT} therefore applies to produce local strong and non-negative solutions to the SKT system.
Let us however note that 
$$\begin{aligned}
\det(A_{\textnormal{SKT}}+\transp{}A_{\textnormal{SKT}})(u_1,u_2) = (d_1+2a_{11}u_1+a_{12}u_2)(d_2+a_{21}u_1+2a_{22}u_2)\\
-(a_{12} u_1+a_{21}u_2)^2
\end{aligned}
$$
may become negative on $\R_{\geq 0}\times \R_{\geq 0}$. For instance for $a_{11}=a_{22}=0$, this expression becomes negative on the two fundamental axes, far from the origin (and therefore also near those axes). This simple example explains why Petrovskii's condition is indeed crucial for the study of parabolic systems and was already pointed out by Amann \cite{amann}.   {Surprisingly enough, for generalizations of the SKT model to the case of multiple populations (more than two species), it has been noticed only recently that a similar analysis can be carried on (see \cite[Section 7]{chenjun2021}) : for an arbitrary number $N$ of population species, the corresponding generalization of the SKT model (see \cite{chenjun2021}) satisfies $A_{\textnormal{SKT}}(\R_{\geq 0}^N)\subset\mathscr{P}$ and our Theorem~\ref{thm:SKT} applies to this setting as well.}

\subsection{Notations}\label{subsec:not}
In the following we denote~${\mathbb P}_0  := \mbox{Id}-\langle \cdot\rangle$ the orthogonal projection  from~$\L^2(\T^d)$  onto mean free functions. For $T>0$, we note $Q_T$ the (periodic) closed cylinder~$Q_T:=[0,T]\times\T^d$. For $1\leq p \leq \infty$ the $\L^p(\T^d)$ and $\L^p(Q_T)$ norms will be noted~$\|\cdot\|_p$ (if there is no ambiguity), while we will generally use $\|\cdot\|_X$ for the norm of some functional space $X$. We shall sometimes use the shorthand notation~$\L^p_T$ for~$\L^p(0,T)$.

  \medskip

  For any real number $s$ we   recall that~$X_T^s$ is the space $\mathscr{C}^0([0,T];\H^s(\T^d))$ and~$Y_T^s$ is the space $\L^2(0,T;\H^{s}(\T^d))$; we then define the energy space $E_T^s:=X_T^s \cap Y_T^{s+1}$ that we equip with the norm $V\mapsto (\|V\|_{X_T^s}^2 + \|\nabla V\|_{Y_T^s}^2)^{1/2}$.

  \medskip
  
For $T>0$, $\alpha\in[0,1]$ and $k\in\N$  we denote by $\mathscr{C}^{k,\alpha}(Q_T)$ the space of $k$ times continuously differentiable functions, whose partial derivatives of order $k$ are $\alpha$-Hölder continuous and we denote by~$\|\cdot\|_{\mathscr{C}^{k,\alpha}(Q_T)}$ the corresponding norm. We simply note $\mathscr{C}^k(Q_T)$ when~$\alpha=0$ and sometimes specify the set of values $X$ writing $\mathscr{C}^{k,\alpha}(Q_T;X)$.

\medskip
  
  We  fix a norm $|\cdot|$ on $\mathbb{C}^N$ and the subordinate norm $\vertiii{\cdot}$ on $\textnormal{M}_N(\mathbb{C})$. For a continuous matrix field $M\in\mathscr{C}^0(Q_T;\mathscr{P})$, $\|M\|_\infty$ will refer to the uniform norm of $M$ (with $\vertiii{\cdot}$ at arrival). For such a matrix field $M$, there exists $\eta(M)>0$ such that $M\in\mathscr{C}^0(Q_T;\mathscr{P}_{\eta(M)})$.  We refer to Appendix Section~\ref{appendix:petrovskii} for the definition and properties of this function $\eta$.
For~$\alpha\in[0,1]$ and a matrix field $M\in\mathscr{C}^{0,\alpha}(Q_T;\mathscr{P})$, we will use repeatedly the following notation 
\begin{equation}\label{defnormalpha}
[M]_{\alpha} := \|M\|_{\mathscr{C}^{0,\alpha}(Q_T)}+\eta(M)^{-1} \, .
\end{equation}
  Finally if~$C_1,\dots,C_n$ is a collection of positive numbers, we   write $$A \lesssim_{C_1,\dots,C_n} B$$ 
if there is an increasing function~$g$   such that
$$
A \leq g(C_1+\dots+C_n) B \, .
$$
 Such a function does not depend on any other relevant variable and it is  liable to change from line to line. We will in general not track it.
\subsection{Main results in the linear setting}
Theorem~\ref{thm:exloc} will be obtained thanks to a detailed study of the linear setting, that is  of  system~(\ref{eq:Pt})    where $M$ is a \emph{given} matrix field. We collect in this paragraph some results which will be useful in the proof of Theorem~\ref{thm:exloc} and Theorem~\ref{thm:lbu} and, even though focusing on the linear setting, are interesting for their own sake. We will often use the notation $L_M$ for the linear differential operator applied to $V$ in the left-hand side of \eqref{eq:Pt}:
   $$
L_M V:=    \partial_t  V - \sum_{k=1}^d \partial_k\big[M\partial_k V  \big] \, .  
   $$
    Well-posedness for~\eqref{eq:Pt} will be established under adequate assumptions on $M$ and thanks to the following \emph{a priori} estimates  in the $\H^s(\T^d)$ setting for $s>d/2$. As will be shown later, one can assume without loss of generality that the functions under study are mean free.
\begin{thm} \label{thm:PetrovskiiHs}
    Let~$T>0$, $\alpha\in(0,1]$, $s>d/2$ and consider a matrix field $M\in\mathscr{C}^{0,\alpha}(Q_T;\mathscr{P})$ which belongs furthermore to $Y_T^{s+1}$. For any $V$ in~$E_T^s$ such that $V(0)\in\H^s(\T^d)$, $L_{M }V\in Y_T^{s-1}$ and $\langle V(t)\rangle=0$ for all $t\in[0,T]$, {one has actually} that $V$ belongs to $E_T^s\cap\mathscr{C}^{0,\alpha_s}(Q_T)$ for some $\alpha_s\in(0,1)$ which depends only on $s$ and
\begin{align}\label{thm2Hsestimate}  
\|  V \|_{E_T^s} + \|V\|_{\mathscr{C}^{0,\alpha_s}(Q_T)} \lesssim_{T,[M]_{\alpha},\|M\|_{Y_T^{s+1}}}  \|V^0\|_{\H^s(\T^d)} +  \|L_{M }V\|_{Y_T^{s-1}}\, .
\end{align}
\end{thm}
 

\begin{rem}\label{rem:no T dependence} The proof Theorem~{\rm\ref{thm:PetrovskiiHs}} may easily be adapted to the case when~$T = \infty$, provided the matrix field~$M$  converges as time goes to infinity towards a stationary matrix field~$\overline M \in  {\mathscr{C}^0}(\T^d;\mathscr{P})$.
\end{rem}
\noindent  Building on this \emph{a priori} estimate, we have the following wellposedness result for the  Cauchy problem~(\ref{eq:Pt}).
\begin{thm}\label{thm:exlin}
  Fix $T>0$, $s>d/2$ and $\alpha\in(0,1]$. For $M\in\mathscr{C}^{0,\alpha}(Q_T;\mathscr{P})\cap Y_T^{s+1}$ and~$(V^0,F) \in \mathcal D_T^s$, the  Cauchy problem~{\rm(\ref{eq:Pt})} is well posed $E_T^s$ and its solution belongs furthermore to $\mathscr{C}^{0,\alpha_s}(Q_T)$, for some $\alpha_s\in(0,1)$.
\end{thm}
 \subsection{Plan of the paper}
In the coming Section~\ref{sec:estimlin} we prove Theorem~\ref{thm:PetrovskiiHs}, which concern \emph{a priori} estimates. This will lead, in Section~\ref{sec:exlin}, to the proof of the linear  wellposedness Theorem~\ref{thm:exlin}. The proof of the nonlinear  wellposedness Theorem~\ref{thm:exloc}
is provided in Section~\ref{sec:road} while Theorem~\ref{thm:lbu} is proved in Section~\ref{prooflbu}. Finally Theorem~\ref{thm:SKT}  is proved in Section~\ref{proofprop:sign}.
Section~\ref{end-point Besov} is devoted to the end-point case in Besov spaces (see Theorem~\ref{thm:exlocbesov}). 
Four appendixes are devoted to some classical results on Sobolev spaces, to basics of Littlewood-Paley theory,  to important facts related to the Petrovskii condition, and to a technical but useful retraction result, of $\R^N$ on $\R_{\geq 0}^N$, respectively.

\section{Estimates in the linear case}\label{sec:estimlin}

We start by studying the case of a matrix field independent of the space variable (see Paragraph~\ref{sct:homogeneous}), first in the constant coefficient case (Proposition~\ref{prop:systlin}), and then in the time-dependent case  (Corollary~\ref{coro:systlinbis}).  We explain then in Paragraph~\ref{subsec:redlip} how the proof of Theorem~\ref{thm:PetrovskiiHs} can be reduced to a simpler result (Lemma~\ref{lem:goalL2:weak}) and then prove this lemma in Subsection~\ref{subsec:proof_lemma}. 

      \subsection{The  case of a constant matrix field}
In this paragraph we treat the simplest case in which the matrix field is constant. Well-posedness is obtained in $E_T^s$ for $s\geq 0$ together with an estimate.
 \begin{prop}\label{prop:systlin}
Fix $\delta>0$ and $B\in\mathscr{P}_\delta$, as well as~$s\in \R$ and $T>0$. For $(V^0,F)\in \mathcal D_T^s$ having both vanishing spatial mean, the Cauchy problem
$$  \left\{
  \begin{aligned}
    \partial_t V-B \Delta V &= F \, ,\\
   V_{|t=0} &= V^0 \, ,
  \end{aligned}
  \right.                                
$$
is well posed in the energy space~$E_{T}^s$ with the following energy estimate
 \begin{align}
   \label{eq:boundlinHs}
   \|V\|_{X_{T}^s}^2 + \delta \|\nabla V\|_{Y_{T}^s}^2 \leq \textnormal{C}_{B,\delta} \Big( \|  V^0\|_{\H^s(\T^d)}^2 + \frac{1}{\delta} \|F\|_{Y_{T}^{s-1}}^2\Big)\, ,
\end{align}
with  \begin{equation}
\label{eq:CAdelta}  \C_{B,\delta} := a_N \Big(1+\frac {\vertiii{B }} \delta\Big)^N \, ,
\end{equation}
and $a_N$ a constant depending only on the dimension $N$ of the system.
\end{prop}
\begin{proof}

The Schur decomposition ensures that there is a unitary matrix~$\mathcal U\in\textnormal{U}_N(\mathbb{C})$  such that~$T:= \mathcal U B \mathcal U^\star$ is upper triangular with diagonal terms~$d_1,\dots,d_N  $ of real part larger than~$ \delta$ and super-diagonal terms denoted~$r_{i,j}$ (and~$R$ is the corresponding matrix).  
Let us set~$\tilde V:= \mathcal U^\star V$ and~$\tilde F:= \mathcal U^\star F$. Then 
$$
   \partial_t   \tilde V-T \Delta  \tilde V  =  \tilde F \, ,
$$
and since~$T$ is upper triangular in particular the last component satisfies 
$$
   \partial_t   \tilde V_N -d_N  \Delta  \tilde V_N  =  \tilde F_N  \, .
$$
An energy estimate in~$\H^s(\T^d)$ on this equation provides directly
$$
\frac12\frac d {dt}\|   \tilde V_N(t)\|_{\H^s(\T^d)} ^2+ \mbox{Re} \, d_N \|\nabla  \tilde V_N\|_{\H^s(\T^d)}^2 \leq     \| \tilde F_N\|_{\H_{T}^{s-1}}\|\nabla \tilde V_N\|_{\H^s(\T^d)}\, , 
$$
hence
  in particular 
$$
\|   \tilde V_N\|_{X_{T}^s}^2 + \delta \|\nabla  \tilde V_N\|_{Y_{T}^s}^2 \lesssim      \|   \tilde V_N(0)\|^2_{\H^s(\T^d)}    + \frac{1}{\delta} \| \tilde F_N\|_{Y_{T}^{s-1}}^2  \, .
$$
Now  we argue by iteration: there holds
$$
   \partial_t   \tilde V_{N-1} -d_{N-1}  \Delta  \tilde V_{N-1} =  \tilde F_{N-1} + r_{N-1,N}\Delta   \tilde V_{N }  $$
   so
   again
   $$
   \begin{aligned}
   \frac12\frac d {dt}\|   \tilde V_{N-1}(t)\|_{\H^s(\T^d)}^2 +\delta \|\nabla  \tilde V_{N-1}\|_{\H^s(\T^d)}^2 &\leq     \| \tilde F_{N-1}\|_{\H_{T}^{s-1}}\|\nabla \tilde V_{N-1}\|_{\H^s(\T^d)}\\
   &\quad  +\vertiii{R} \|\nabla  \tilde V_{N-1}\|_{\H^s(\T^d)} \|\nabla  \tilde V_{N}\|_{\H^s(\T^d)}\, ,
   \end{aligned}
  $$
 and finally 
$$
\begin{aligned}
 \|   \tilde V_{N-1}(t)\|_{X_{T}^s}^2 + \delta  \|   \tilde V_{N-1}\|_{Y_{T}^s}^2 \lesssim     \|   \tilde V_{N-1}(0)\|^2_{\H^s(\T^d)}  + \frac{1}{\delta} \| \tilde F_{N-1}\|_{Y_{T}^{s-1}}^2 \\
   +  \frac {\vertiii{R} ^2}{\delta} \|\nabla  \tilde V_N\|_{Y_{T}^s}^2 \, .
\end{aligned}
$$
It follows that
$$
\begin{aligned}
 \|   \tilde V_{N-1}\|_{X_{T}^s}^2 + \delta  \|  \nabla \tilde V_{N-1}\|_{Y_{T}^s}^2 \lesssim  \|   \tilde V_{N-1}(0)\|^2_{\H^s(\T^d)} +  \frac {\vertiii{R} ^2}{\delta^2}\|   \tilde V_{N}(0)\|^2_{\H^s(\T^d)}  \\
 +  \frac{1}{\delta} \| \tilde F_{N-1}\|_{Y_{T}^{s-1}}^2 +  \frac {\vertiii{R} ^2}{\delta^3}\| \tilde F_{N}\|_{Y_{T}^{s-1}}^2 \, .
\end{aligned}
$$
Arguing similarly at each step gives finally
$$
\|   \tilde V \|_{X_{T}^s}^2  +\delta  \| \nabla   \tilde V \|_{Y_{T}^s}^2 \leq  \Big(1+ \frac{\vertiii{R} }{\delta}\Big)^N\Big( \|   \tilde V (0)\|_{\H^s(\T^d)} + \frac1\delta  \| \tilde F \|_{Y_{T}^{s-1}}^2 \Big) \, .
$$ 
Estimate \eqref{eq:boundlinHs} is proved.
\end{proof}
\subsection{The case of a homogeneous in space matrix field}\label{sct:homogeneous} 
In this paragraph we focus on the case when $M$ does not depend on the space variable but may depend on   time: using Proposition~\ref{prop:systlin}, we can actually indeed recover a similar result  for a class of non autonomous systems.  \begin{coro}\label{coro:systlinbis}
  Let $s\in\R$, $\alpha>0$ and $M\in {\mathscr{C}^{0,\alpha}}([0,T] ;\mathscr{P})$. For any $V\in E_T^s$ having vanishing spatial mean at all times and such that~$L_{M }V\in Y_T^{s-1}$  there holds 
\begin{align}\label{eq:estimate non autonomous}
  \|V\|_{E_T^s} \lesssim_{T,[M]_\alpha} \|V (0)\|_{\H^s(\T^d)} + \|L_{M }V\|_{Y_T^{s-1}}\, .
\end{align}
 \end{coro}
\begin{proof}
  We consider a subdivision $t_0=0<t_1<\cdots < t_\kappa = T$ of $[0,T]$, such that each subinterval has size smaller than~$T/\kappa$ with~$\kappa$ to be determined. Using the notation introduced in Corollary~\ref{coro:thetapet}, we see that each matrix $M(t_i)$ belongs to $\mathscr{P}_{\eta(M)}$. So writing
\[\partial_t V - M(t_i)\Delta V =  L_{M }V + \big(M-M(t_i)\big)\Delta V\, ,\]
  we get from Proposition~\ref{prop:systlin} for $t\in[t_i,t_{i+1}]$, shifting the initial time to $t_i$ 
\begin{multline*}
  \|V\|_{\L^\infty([t_i,t];\H^s(\T^d))}^2  + \eta(M)\int_{t_i}^{t} \|\nabla V(t')\|_{\H^s(\T^d)}^2\,\dd t'\leq \textnormal{C}_{M(t_i),\eta(M)} \|V(t_i)\|_{\H^s(\T^d)}^2  \\ + \frac{\textnormal{C}_{M(t_i),\eta(M)}}{\eta(M)} \int_{t_i}^t\|L_{M }V(t')\|_{\H^{s-1}(\T^d)}^2\,\dd t' \\+ \frac{\textnormal{C}_{M(t_i),\eta(M)}}{\eta(M)} \int_{t_i}^t \vertiii{M(t')-M(t_i)}^2 \|\nabla V(t')\|_{\H^s(\T^d)}^2\,\dd t'.
\end{multline*}
Now, returning to the definition \eqref{eq:CAdelta} of $\textnormal{C}_{B,\delta}$ and  recalling notation~(\ref{defnormalpha}), we can rewrite the previous inequality as
$$\begin{aligned}
  \|V\|_{\L^\infty([t_i,t];\H^s(\T^d))}^2 & + \int_{t_i}^{t} \|\nabla V(t')\|_{\H^s(\T^d)}^2\,\dd t'  \lesssim_{[M]_\alpha} \|V(t_i)\|_{\H^s(\T^d)}^2\\
  &\quad  + \int_{t_i}^t\|L_{M }V(t')\|_{\H^{s-1}(\T^d)}^2\,\dd t' \\
  & \qquad + \int_{t_i}^t \vertiii{M(t')-M(t_i)}^2 \|\nabla V(t')\|_{\H^s(\T^d)}^2\,\dd t' \, .
\end{aligned}$$
We recall that~$\lesssim_{[M]_\alpha}$ stands for multiplication by~$g([M]_\alpha)$ with~$g$ increasing.
  If we choose~$ \kappa$ large enough so that
\begin{align}\label{ineq:L}
\left(\frac{T}\kappa\right)^{2\alpha} \|M\|_{\mathscr{C}^{0,\alpha}([0,T])}^2 < \frac{1}{2g([M]_\alpha)} \, ,
\end{align}
then recalling that~$|t_i-t_{i+1 }| \leq  T/\kappa $ for all $i\in\llbracket 0,\kappa-1\rrbracket$, we have in particular 
\begin{align*}
\forall  i\in\llbracket 0,\kappa-1\rrbracket \, , \quad \sup_{t \in [t_i,t_{i+1}]} \vertiii{M(t)-M(t_i)}^2 \leq \frac{1}{2g([M]_\alpha)} \, \cdotp
\end{align*}
The inequality on $[t_i,t_{i+1}]$ becomes
\begin{multline*}
    \|V\|_{\L^\infty([t_i,t];\H^s(\T^d))}^2  + \frac12 \int_{t_i}^{t} \|\nabla V(t')\|_{\H^s(\T^d)}^2\,\dd t'\\\lesssim_{[M]_\alpha} \|V(t_i)\|_{\H^s(\T^d)}^2 + \int_{t_i}^t\|L_{M }V(t')\|_{\H^{s-1}(\T^d)}^2\,\dd t' \, .
\end{multline*}
Summing this estimate with $i=0$ and $t=t_1$ with the one for $i=1$ and $t\in[t_1,t_2]$ we get in particular
\begin{multline*}
  \|V(t_1)\|_{\H^s(\T^d)}^2 +   \|V\|_{\L^\infty([t_1,t];\H^s(\T^d))}^2  + \frac12 \int_{0}^{t} \|\nabla V(t')\|_{\H^s(\T^d)}^2\,\dd t'\\\lesssim_{[M]_\alpha} \|V(0)\|_{\H^s(\T^d)}^2 + \|V(t_1)\|_{\H^s(\T^d)}^2 +  \int_{0}^t\|L_{M }V(t')\|_{\H^{s-1}(\T^d)}^2\,\dd t' \, .
\end{multline*}
Recalling that  $\lesssim_{[M]_\alpha}$ stands for multiplication by~$\textnormal{K}_M:=g([M]_\alpha)$, we have therefore 
\begin{multline*}
\|V\|_{\L^\infty([t_1,t];\H^s(\T^d))}^2 + \frac12 \int_{0}^{t} \|\nabla V(t')\|_{\H^s(\T^d)}^2\,\dd t'\\\leq \textnormal{K}_{M}\|V(0)\|_{\H^s(\T^d)}^2 + (\textnormal{K}_{M}-1) \|V(t_1)\|_{\H^s(\T^d)}^2 + \textnormal{K}_{M} \int_{0}^t\|L_{M }V(t')\|_{\H^{s-1}(\T^d)}^2\,\dd t',
\end{multline*}
which implies eventually on $[0,t_2]$ 
\begin{multline*}
\|V\|_{\L^\infty([0,t];\H^s(\T^d))}^2 + \frac12 \int_{0}^{t} \|\nabla V(t')\|_{\H^s(\T^d)}^2\,\dd t'\\\leq \textnormal{K}_{M}^2\|V(0)\|_{\H^s(\T^d)}^2 + \textnormal{K}_{M} \int_{0}^t\|L_{M }V(t')\|_{\H^{s-1}(\T^d)}^2\,\dd t'.
\end{multline*}
Iterating, we recover on $[0,t_\kappa] = [0,T]$
$$
\begin{aligned}
\|V\|_{X_T^s}^2 + \frac12 \int_{0}^{T} \|\nabla V(t')\|_{\H^s(\T^d)}^2\,\dd t' & \leq \textnormal{K}_{M}^\kappa\|V(0)\|_{\H^s(\T^d)}^2\\
&\qquad  + \textnormal{K}_{M} \int_{0}^T\|L_{M }V(t')\|_{\H^{s-1}(\T^d)}^2\,\dd t'.
\end{aligned}
$$
The proof is over once noticed that the condition \eqref{ineq:L} required on $\kappa$ can indeed be satisfied choosing $\kappa=\widetilde g(T + [M]_\alpha)$ with $\widetilde g$ some increasing function. 
\end{proof}

\subsection{Reduction of Theorem~\ref{thm:PetrovskiiHs} to a single lemma}\label{subsec:redlip}
In this subsection we explain how the estimate of Theorem~\ref{thm:PetrovskiiHs} can be recovered by the following (seemingly) weaker result.
\begin{lem}\label{lem:goalL2:weak}
Fix $T>0$, $s>d/2$, $\alpha\in(0,1)$ and $M\in\mathscr{C}^{0,\alpha}(Q_T;\mathscr{P})\cap Y_T^{s+1}$ a matrix field which is furthermore assumed to belong to $X_T^{s+1}$. For any $V\in E_T^s$ such that  $L_M V\in Y_T^{s-1}$ and $\langle V(t)\rangle = 0$ for all $t\in[0,T]$, one has 
\begin{align*}
  \|V\|_{E_T^s} \lesssim_{T,[M]_\alpha}  \|V(0)\|_{\H^s(\T^d)} + \| L_M V \|_{Y_T^{s-1}} + \Big(1+\|M\|_{X^{s+1}_T}\Big)\|V\|_{Y_T^s}.
\end{align*}
 \end{lem}
 Admitting for the moment the previous lemma, Theorem~\ref{thm:PetrovskiiHs} can be proved thanks to an approximation argument. First notice that owing to Lemma~\ref{lem:sobo2}, we only need to prove the $E_T^s$ part of the estimate as the Hölder one follows then immediately.

 \medskip
 
 If $M\in\mathscr{C}^{0,\alpha}(Q_T;\mathscr{P})\cap Y_T^{s+1}$, usual convolution properties lead to the existence of lipschitz matrix-valued functions $(M_\ep)_\ep$  for which
\begin{align}
  \label{ineq:Mep1}\|M_\ep\|_{\mathscr{C}^{0,\alpha}(Q_T)}&\leq \|M\|_{\mathscr{C}^{0,\alpha}(Q_T)}\, ,\\
  \label{ineq:Mep2}\|M-M_\ep\|_\infty&\leq \|M\|_{\mathscr{C}^{0,\alpha}(Q_T)}\ep^\alpha\, ,\\
    \label{ineq:Mep2bis}\|M_\ep\|_{Y_T^{s+1}}&\leq \|M\|_{Y_T^{s+1}}\, ,\\
  \label{ineq:Mep3}\|M_\ep\|_{X_T^{s+1}}&\leq (1+\ep^{-(s+1)})\|M\|_{\infty}\, .
 \end{align}
 Because of \eqref{ineq:Mep2} and the continuity of $\eta$ (see Paragraph~\ref{subsec:not} for the definition of~$\eta$, and Corollaries~\ref{coro:thetapet} and \ref{coro:etahold} for its properties), for $\ep$ small enough we have $\eta(M_\ep)\geq \eta(M)/2$ so that~$[M_\ep]_\alpha\lesssim[M]_\alpha$ and we thus can infer from Lemma~\ref{lem:goalL2:weak} that for any $V\in E_T^s$ such that~$L_{M _\ep}V\in Y_T^{s-1}$ and $\langle V(t)\rangle = 0$ for all $t\in[0,T]$,
\begin{align}\label{estimabove}
  \|V\|_{E_T^s}    \lesssim_{T,[M]_\alpha}  \|V(0)\|_{\H^s(\T^d)} +\| L_{M_\ep }V \|_{Y_T^{s-1}} 
+ \Big(1+\|M_\ep\|_{X_T^{s+1}}\Big)\|V\|_{Y_T^s} \, .
\end{align}
Now since
\begin{align*}
L_{M_\ep }V-L_{M }V = \sum_{k=1}^d \partial_k \big[(M_\ep-M)\partial_k V\big] \, ,
\end{align*}
we have thanks to \eqref{ineq:Mep2} and the product rule given in Proposition~\ref{prop:rhsSobo}  \begin{multline*}\|(L_{M_\ep}-L_M)V\|_{Y_T^{s-1}}^2 \lesssim  \ep^{2\alpha}  \|M\|_{\mathscr{C}^{0,\alpha}(Q_T)}^2 \int_0^T \|\nabla V(t)\|_{\H^s(\T^d)}^2\,\dd t \\+ \int_0^T \|(M_\ep-M)(t)\|_{\H^{s+1}(\T^d)}^2\|V(t)\|_{\H^s(\T^d)}^2\,\dd t\,,\end{multline*} so the estimate \eqref{estimabove} becomes, using \eqref{ineq:Mep3}
$$
\begin{aligned}
  \|V\|_{X_T^s}^2 +   \|\nabla V\|_{Y_T^s}^2 & \lesssim_{T,[M]_\alpha} \|V(0)\|_{\H^s(\T^d)} ^2+\| L_{M}V \|_{Y_T^{s-1}}^2 \\
  &\quad  + \big(1+ \ep^{-2(s+1)}\big)\int_0^T \|V(t)\|_{\H^s(\T^d)}^2\,\dd t   +\ep^{2\alpha}  \int_0^T \|\nabla V(t)\|_{\H^s(\T^d)}^2\,\dd t \\
 &\qquad  + \int_0^T \|(M_\ep-M)(t)\|_{\H^{s+1}(\T^d)}^2\|V(t)\|_{\H^s(\T^d)}^2\,\dd t\,.
\end{aligned}
$$
Since the multiplicative constant behind~$ \lesssim_{T,[M]_\alpha}$ is an increasing function~$  g(T+[M]_\alpha)$, if we take $\ep$ small enough so as
\begin{align}\label{ineq:fix:ep}
                                      g(T+[M]_\alpha)\ep^{2\alpha}<\frac12 \, ,
\end{align}
the previous estimate implies
\begin{multline}\label{ineq:1+Mep}
  \|V\|_{X_T^s}^2 +   \|\nabla V\|_{Y_T^s}^2  \lesssim_{T,[M]_\alpha} \|V(0)\|_{\H^s(\T^d)} ^2+\| L_{M}V \|_{Y_T^{s-1}}^2\\   + \big(1+ \ep^{-2(s+1)}\big)\int_0^T \|V(t)\|_{\H^s(\T^d)}^2\,\dd t  \\
 + \int_0^T \|(M_\ep-M)(t)\|_{\H^{s+1}(\T^d)}^2\|V(t)\|_{\H^s(\T^d)}^2\,\dd t\,.
\end{multline}
Recalling the definition of $X_T^s$ and using \eqref{ineq:Mep2bis} we are just off one Grönwall lemma of ending the proof of Theorem~\ref{thm:PetrovskiiHs}, provided $\ep$ can be replaced by some decreasing function of $T+[M]_\alpha$: let us track the precise dependence of $\ep$ with respect to $M$. The only two conditions on $\ep$ are sufficient smallness for \eqref{ineq:fix:ep} to hold, and  for~$\eta(M_\ep)\geq \eta(M)/2$ to be satisfied. For the first condition, it is clearly satisfied if $\ep$ is replaced by some decreasing function of $T+[M]_\alpha$. The second condition is trickier. As $\|M\|_\infty+\eta(M)^{-1}\leq [M]_\alpha$, we rely on Corollary~\ref{coro:etahold} to infer the existence of an non-increasing function $f$ such that \[\ep\leq f([M]_\alpha) \Rightarrow\eta(M_\ep)\geq \eta(M)/2,\] and we can thus replace $\ep$ by some non-increasing function of $[M]_\alpha$ in the previous computations. Theorem~\ref{thm:PetrovskiiHs} is proved. \qed

\subsection{Proof of Lemma~\ref{lem:goalL2:weak}}\label{subsec:proof_lemma} 
In this subsection we prove Lemma~\ref{lem:goalL2:weak} which, due to the argument of Subsection~\ref{subsec:redlip} implies Theorem~\ref{thm:PetrovskiiHs}. The idea  is to  reduce to the case of a constant in space matrix field considered in  Paragraph~\ref{sct:homogeneous}, namely Corollary~\ref{coro:systlinbis}, by a partition of unity of~$   \T^d$. 

\medskip

We start with a localization lemma.

 \begin{lem}\label{lem:loc}
Consider the assumptions of Lemma~\ref{lem:goalL2:weak}. For~$\ep$ small enough, and depending decreasingly on~$T+[M]_\alpha$, the following holds. For any $V\in E_T^s$ such that $L_{M }V $ belongs to~$Y_T^{s-1}$ and any smooth bump function $\theta$ supported in a ball of~$\T^d$ of radius~$\ep$, there holds 
\begin{align*}
    \|\mathbb{P}_0(\theta V)\|_{E_T^s}  \lesssim_{T,[M]_\alpha}\|\theta\|_{\H^{m}(\T^d)}\Big[\|\mathbb{P}_0(\theta V)(0)\|_{\H^s(\T^d)}  +  \|    L_{M }V\|_{Y_T^{s-1}} + \|M\|_{X_T^{s+1}} \|   V  \|_{Y_T^s}\Big],
\end{align*}
where $\mathbb{P}_0$ is the $\L^2(\T^d)$ projection onto vanishing mean functions, $\lesssim_{T,[M]_\alpha}$ a symbol depending only increasingly in $T+[M]_\alpha$ and $m$ a natural integer depending only on $s$ and~$d$.
      \end{lem}
      \begin{proof}
For the moment, let us start the computation with an arbitrary $\ep>0$ (yet to be fixed) with $\theta$ supported in the ball of center $x^\star \in \T^d$ and radius $\ep>0$. For $M^\star : t\mapsto M(t,x^\star)$ we have for all $t\in[0,T]$
 \begin{equation}\label{PstarP} 
 \forall x \in \mbox{Supp}\,  \theta \, , \quad \vertiii{M(t,x)-M^\star(t)} \leq \ep^\alpha \|M\|_{\mathscr{C}^{0,\alpha}(Q_T)}\,.
 \end{equation}
 Next  we compute
  \begin{align} \label{eq:defF}
L_{M^\star}(\theta V) := \partial_t( \theta V)-  M^\star \Delta (\theta V)= \theta  L_{M}V +  \sum_{k=1}^d \big( \partial_kR_k +      S_k  \big)\,,
  \end{align}
  where we noted
  \begin{align}
\label{eq:Rk}    R_k &:= (M-M^\star)\partial_k(\theta V),\\
\label{eq:Sk}    S_k &:= \stackrel{S_{k,1}}{\overbrace{-2 \partial_k (M V \partial_k \theta) }} + \stackrel{S_{k,2}}{\overbrace{ V \partial_k M \partial_k \theta+MV \partial^2_k \theta }}.
  \end{align}
Noticing $\mathbb{P}_0L_{M^\star} = L_{M^\star}\mathbb{P}_0$ and using Corollary~\ref{coro:systlinbis} we have therefore
\begin{align*}
  \|\mathbb{P}_0(\theta V)\|_{E_T^s} &\lesssim_{T,[M]_\alpha} \|\mathbb{P}_0(\theta V)(0)\|_{\H^s(\T^d)} + \|\mathbb{P}_0 L_{M^\star}(\theta V)\|_{Y_T^{s-1}}\\
  &\lesssim_{T,[M]_\alpha} \|\mathbb{P}_0(\theta V)(0)\|_{\H^s(\T^d)} + \|\mathbb{P}_0(\theta L_M V)\|_{Y_T^{s-1}} + \sum_{k=1}^d \|R_k\|_{Y_T^s} \\
  & \qquad\qquad\qquad  +\|\mathbb{P}_0S_k\|_{Y_T^{s-1}}.
\end{align*}
Now, as $\theta$ is supported on a ball $\B$ of radius $\ep$, if $\widetilde{\theta}$ is another $[0,1]$-valued bump function equalling $1$ on $\B$ and vanishing outside the ball $\widetilde{\B}$ of same center but twice radius, we infer from the definition \eqref{eq:Rk} of $R_k$ and the product rule Proposition~\ref{prop:rhsSobo}
\begin{align*}
  \|R_k\|_{Y_T^s} &= \|\widetilde{\theta}(M-M^\star)\partial_k(\theta V)\|_{Y_T^s} \\
                   &\lesssim \ep^{\alpha} \|M\|_{\mathscr{C}^{0,\alpha}(Q_T)}\|\partial_k(\theta V)\|_{Y_T^s} + \|\widetilde{\theta}(M-M^\star)\|_{X_T^{s+1}}\|\theta V\|_{Y_T^s},
\end{align*}
where we used estimate \eqref{PstarP} in the second line. Now, as $\widetilde{\theta}$ equals $1$ on $\B$ and vanishes outside $\widetilde{\B}$, there holds $\|\widetilde{\theta}\|_{\H^{s+1}(\T^d)} = f(\ep^{-1})$ where $f$ is some non-decreasing function that we do not compute. Using the algebra structure of $\H^{s+1}(\T^d)$, we infer
\vspace{-0.3cm}
\begin{multline*}
  \|\mathbb{P}_0(\theta V)\|_{E_T^s} \lesssim_{T,[M]_\alpha} 
 \|\mathbb{P}_0(\theta V)(0)\|_{\H^s(\T^d)} + \| \mathbb{P}_0(\theta L_{M} V)\|_{Y_T^{s-1}}  + \sum_{k=1}^d \|\mathbb{P}_0S_k\|_{Y_T^{s-1}}\\+ \ep^\alpha \|\nabla(\theta V)\|_{Y_T^s} + f(\ep^{-1})\|M\|_{X_T^{s+1}}\|\theta V\|_{Y_T^s}.
\end{multline*}
At this point, we fix the smallness of $\ep$ so as to absorb $\|\nabla(\theta V)\|_{Y_T^s}$ in the norm $\|\mathbb{P}_0(\theta V)\|_{E_T^s}$ of the left hand side. As the symbol $\lesssim_{T,[M]_\alpha}$ refers to some non-decreasing function of~$T+[M]_\alpha$, this procedure can indeed be carried on with a choice of $\ep$ as a non-increasing function of~$T+[M]_\alpha$. We have therefore
\begin{multline*}
  \|\mathbb{P}_0(\theta V)\|_{E_T^s} \lesssim_{T,[M]_\alpha} 
 \|\mathbb{P}_0(\theta V)(0)\|_{\H^s(\T^d)} + \|\mathbb{P}_0(\theta L_{M} V)\|_{Y_T^{s-1}}  \\ + \|M\|_{X_T^{s+1}}\|\theta V\|_{Y_T^s} +\sum_{k=1}^d \|\mathbb{P}_0S_k\|_{Y_T^{s-1}}.
\end{multline*}
Owing to the algebra structure of  $\H^s(\T^d)$, we have  \[\|\mathbb{P}_0(\theta V)(0)\|_{\H^s(\T^d)} + \|\theta V\|_{Y_T^s} \lesssim \|\theta\|_{\H^s(\T^d)} \Big( \|V(0)\|_{\H^s(\T^d)} + \|V\|_{Y_T^s}\Big),\] and for an integer $m$ large enough (see for instance the proof of \cite[Theorem 1.62]{BCD}) depending only on~$s$ and $d$ we have
\[\|\mathbb{P}_0(\theta L_M V)\|_{Y_T^{s-1}} \lesssim \|\theta\|_{\H^{m}(\T^d)} \|L_M V\|_{Y_T^{s-1}}.\]
At the end of the day, it only remains to handle the $S_k$ terms to conclude. Turning back to the decomposition $S_k = S_{k,1}+S_{k,2}$ given in \eqref{eq:Sk}, we note on the one hand that due to  the algebra structure of $\H^s(\T^d)$, there holds
\begin{align*}
  \|\mathbb{P}_0S_{k,1}\|_{Y_T^{s-1}} &\lesssim \|MV \partial_k\theta\|_{Y_T^s}   \\
  &\lesssim \|\theta\|_{\H^{s+1}(\T^d)} \|M\|_{X_T^{s}} \|V\|_{Y_T^s}\, .
\end{align*}
For the $0$ order terms, we can just write $\|\mathbb{P}_0S_{k,2}\|_{Y_T^{s-1}} \leq \|S_{k,2}\|_{Y_T^s}$ and invoke once more the algebra structure of $\H^s(\T^d)$ to infer
\[\|S_{k,2}\|_{Y_T^s} \lesssim \|\theta\|_{\H^{s+2}(\T^d)}\|M\|_{X_T^{s+1}} \|V\|_{Y_T^s}\, .\] 
The proof of Lemma~\ref{lem:loc} is over. $\qedhere$
\end{proof}
We now proceed to the proof of Lemma~\ref{lem:goalL2:weak}. We fix~$\ep>0$ as in Lemma~\ref{lem:loc} and decompose~$\T^d $ into a finite union of  essentially disjoint hypercubes denoted~$(K^{j}_\ep)_{1 \leq j \leq J_M} $  centered at  points~$ x_j \in    K^j_\ep $, with sidelengths~$\ep$.  This implies  that~$J_M$   is of the order of~$1/\ep^d$. We then consider a partition of unity~$(\theta^{j}_\ep)_{1 \leq j \leq J_M} $  where each~$\theta^{j}_\ep $ is compactly supported in a ball~$B^{j}_\ep$ of~$\T^d$  of radius~$\ep$ containing strictly~$K^j_\ep $, and takes it values in~$[0,1]$.
 We assume in particular that
for any multi-index~$\alpha  \in \N^d$, there is a   constant~$C_\alpha$  such that  for any~$1 \leq j \leq J_M$, any~$\ep>0$ and any~$t \geq 0$
 \begin{equation}\label{size derivatives}
     \| D_x^\alpha\theta^{j}_\ep\|_{\infty } \leq C_\alpha{\ep^{-|\alpha|}}  \, .
\end{equation} 
We first invoke the local estimate Lemma~\ref{lem:loc} for each of these bump functions   $\theta_{\ep}^j$. We have therefore, for some integer $m\in\N$ 
  \begin{align*}
    \| \mathbb{P}_0(\theta_\ep^j V)\|_{E_T^s} \lesssim_{T,[M]_\alpha} \|\theta_\ep^j \|_{\H^m(\T^d)} \Big[\|V(0)\|_{\H^s(\T^d)} + \|L_M V\|_{Y_T^{s-1}} +  \|M\|_{X_T^{s+1}}\|V\|_{Y_T^s}\Big],
  \end{align*}
for $\lesssim_{T,[M]_\alpha}$ a non-decreasing function of $T+[M]_\alpha$.  Thanks to \eqref{size derivatives}, since $\ep$ is chosen as non-increasing function of $T+[M]_\alpha$, we can actually simply the previous estimate and write
    \begin{align}\label{ineq:tosum}
    \| \mathbb{P}_0(\theta_\ep^j V)\|_{E_T^s} \lesssim_{T,[M]_\alpha} \|V(0)\|_{\H^s(\T^d)} + \|L_M V\|_{Y_T^{s-1}} +  \|M\|_{X_T^{s+1}}\|V\|_{Y_T^s}.
    \end{align}
As $\mathbb{P}_0$ is linear and $\mathbb{P}_0 V = V$, we have for all $t\in[0,T]$
    \begin{align*}
\|V\|_{E_T^s} = \left\|\sum_{j=1}^{J_M} \mathbb{P}_0(\theta_\ep^j V)\right\|_{E^s_T} \leq \sum_{j=1}^{J_M} \|\mathbb{P}_0(\theta_\ep^j V)\|_{E_T^s}.
    \end{align*}
Summing over $j\in\llbracket 1,J_M\rrbracket$ inequalities \eqref{ineq:tosum} we recover
    \begin{align*}
\|V\|_{E_T^s} \lesssim_{T,[M]_\alpha} J_M \Big[\|V(0)\|_{\H^s(\T^d)} + \|L_M V\|_{Y_T^{s-1}} +  \|M\|_{X_T^{s+1}}\|V\|_{Y_T^s}\Big],
    \end{align*}
which ends the proof of Lemma~\ref{lem:goalL2:weak} because $J_M\sim \ep^{-d}$, with $\ep$ a non-increasing function of~$T+[M]_\alpha$.  
 
\subsection{A useful corollary}\label{subsec:useful}
We end this series of \emph{a priori} estimates with a corollary of Theorem~\ref{thm:PetrovskiiHs} which will be useful when $M$ is of the form $A(U)$ with $A:\R^N\rightarrow\mathscr{P}$ a smooth matrix field, as this is the case in Theorem~\ref{thm:exloc}. Having estimate \eqref{thm2Hsestimate} in mind, it is natural to introduce for any $s>d/2$ and $T>0$ the following space \begin{align}\label{def:GTs}
G_T^s := E_T^s \cap \mathscr{C}^{0,\alpha_s}(Q_T).
     \end{align}                                                                           We have then, the following result.
\begin{coro}\label{coro:useful}
Fix $s>d/2$ and $T>0$. Consider $U:Q_T\rightarrow \R^N$ belonging to $G_T^{s}$ and $A:\R^N\rightarrow\mathscr{P}$ a smooth matrix field. For any $V$ in~$E_T^s$ such that $V(0)\in\H^s(\T^d)$, $L_{A(U)}V\in Y_T^{s-1}$ and $\langle V(t)\rangle=0$ for all $t\in[0,T]$, one has actually that $V$ belongs to $G_T^s$ and
\begin{align}\label{thm2Hsestimate:useful}  
\|V\|_{E_T^s} + \|V\|_{\mathscr{C}^{0,\alpha_s}(Q_T)} \lesssim_{T,\|U\|_{G_T^s}}  \|V^0\|_{\H^s(\T^d)} +  \|L_{A(U)}V\|_{Y_T^{s-1}} \, ,
\end{align}
where the symbol $\lesssim_{T,\|U\|_{G_T^s}}$ is a non-decreasing function of $T+\|U\|_{G_T^s}+\|A(0)\|$. 
 \end{coro}
\begin{proof}
  Of course the proof reduces to justifiying the use of Theorem~\ref{thm:PetrovskiiHs} for $M=A(U)$ and to replace the dependence $T+[A(U)]_\alpha+\|A(U)\|_{Y_T^{s+1}}$ by the above  one for the symbol $\lesssim_{T,\|U\|_{G_T^{s+1}}} $ appearing in the estimate.

  \medskip

  We first note that $A$ being smooth, it stabilizes Sobolev spaces and induces locally a lipschitz map. More precisely, since $U\in E_T^s\hookrightarrow\L^\infty(Q_T)$, we use Lemma~\ref{lem:lipsob} with $\Phi=A$, $f=U$, $g=0$ and $\sigma=s+1$ to recover $A(U)\in Y_T^{s+1}$ with a bound
                              \begin{align}\label{ineq:A(U)}
                                \|A(U)-A(0)\|_{Y_T^{s+1}} \lesssim_{\|U\|_{G_T^s}} \|U\|_{Y_T^{s+1}} \, .
                              \end{align}
                          Also, since $A$ is locally lipschitz, the $\alpha$-Hölder regularity of $U$ is inherited by $A(U)$ with an estimate of the form $\|A(U)\|_{\mathscr{C}^{0,\alpha}(Q_T)}\leq g(\|U\|_{\mathscr{C}^{0,\alpha}(Q_T)})$, with $g$ an increasing function. We are now in position to invoke Theorem~\ref{thm:PetrovskiiHs} which states exactly
                              \begin{align*}
\|  V \|_{G_T^s} \lesssim_{T,[A(U)]_\alpha,\|A(U)\|_{Y_T^{s+1}}}  \|V^0\|_{\H^s(\T^d)} +  \|L_{A(U)}V\|_{Y_T^{s-1}} \, .
\end{align*}
Recalling the definition of $[A(U)]_\alpha$ in Paragraph~\ref{subsec:not}, we only need to handle $\eta(A(U))^{-1}$. For this, we use Corollary~\ref{coro:thetapet} to see that $A$ maps the ball of radius $R$ of $\R^N$ into some $\mathscr{P}_{\delta_R}$ with $\delta_R$ decreasing in $R$, so that $\eta(A(U))^{-1}$ is indeed bounded by some increasing function of $\|U\|_\infty\leq \|U\|_{\mathscr{C}^{0,\alpha}(Q_T)}$. Corollary~\ref{coro:useful} is proved.                         \end{proof}

\section{Existence theory and parabolic regularization in the linear case}\label{sec:exlin}
In this short section we first prove Theorem~\ref{thm:exlin} thanks to the \emph{a priori} estimates Theorem~\ref{thm:PetrovskiiHs} established in Section~\ref{sec:estimlin}, and then state and prove as a corollary of these results a propagation of regularity result.

\subsection{Proof of Theorem~\ref{thm:exlin}}

Uniqueness is a straightforward consequence of estimate \eqref{thm2Hsestimate}, so we focus only on the existence part starting. Without loss of generality we only need to establish this existence result replacing $F$ by $\mathbb{P}_0 F$ and $V^0$ by $\mathbb{P}_0 V^0$: if a solution is built in this vanishing mean setting, adding to it
\[\langle V^0\rangle + \int_0^t \langle F(t')\rangle\,\dd t',\]
we will recover a solution in the general case.

The set~$\mathscr{C}^{0,\alpha}(Q_T;\mathscr{P})$ is path-connected for the $\mathscr{C}^{0,\alpha}(Q_T;\textnormal{M}_N(\R))$ topology as is startshaped with respect to the constant identity matrix map.

\medskip

Define $\mathscr{S}$ as the subset of $\mathscr{C}^{0,\alpha}(Q_T;\mathscr{P})\cap Y_T^{s+1}$ constituted of all those matrix-fields for which the problem \eqref{eq:Pt} (for arbitrary data~$(V^0,F)\in\mathcal{D}^s_T$ with vanishing mean) has a solution in $E_T^s$. The set~$\mathscr{S}$ is closed in~$\mathscr{C}^{0,\alpha}(Q_T;\mathscr{P})\cap Y_T^{s+1}$. Indeed, should~$(M_k)_k\in\mathscr{S}^\N$ converge uniformly to $M\in\mathscr{C}^{0,\alpha}(Q_T;\mathscr{P})$, this is already sufficient to ensure (see Corollary~\ref{coro:thetapet}) that $(\eta(M_k))_k$ converges to $\eta(M)>0$, so that the whole sequence satisfies~$\eta(M_k)\geq \delta$ for some $\delta>0$. Now, as $(M_k)_k$ is bounded in $\mathscr{C}^{0,\alpha}(Q_T;\mathscr{P})$ and~$\underline{\lim}_k \eta(M_k)>0$, we infer that~$\overline{\lim}_k [M_k]_\alpha<+\infty$. We thus infer uniformity in $k$ for the \emph{a priori} estimate \eqref{thm2Hsestimate} of Theorem~\ref{thm:PetrovskiiHs} satisfied by the solutions~$V_k$ associated with~$M_k$ (such solutions $V_k$ exist precisely because the sequence $(M_k)_k$ lies in $\mathscr{S}^\N$). The equation being linear, we recover in this way by a weak($-\star$) compactness argument first a solution in $\L^\infty(0,T;\H^s(\T^d))\cap Y_T^{s+1}$ which in fact belongs to $E_T^s$, using the equation to control the time derivative and the standard Lemma~\ref{lem:sobo1}. The subset~$\mathscr{S}$ is also open in~$\mathscr{C}^{0,\alpha}(Q_T;\mathscr{P})\cap Y_T^{s+1}$. Indeed, for $M\in\mathscr{S}$ and $\ep>0$ to be defined later, let's consider $\widetilde{M}\in\mathscr{C}^{0,\alpha}(Q_T;\mathscr{P})\cap Y_T^{s+1}$ such that \[\|M-\widetilde{M}\|_\infty +\|M-\widetilde{M}\|_{Y_T^{s+1}}< \ep.\] One can define the map from $E_T^s$ to itself which sends $V$ to the solution~$\widetilde{V}$ of (taking $V^0$ as initial data)
\begin{align*}
    \partial_t \widetilde{V}-\sum_{k=1}^d \partial_k \big[M\,\partial_k \widetilde{V}\big] = F+\sum_{k=1}^d \partial_k\big[(M-\widetilde{M})\partial_k V\big].
\end{align*}
The existence of~$\widetilde{V}$ is due to the fact that $M$ has been chosen in~$\mathscr{S}$. Linearity and the \emph{a priori} estimate \eqref{thm2Hsestimate} of Theorem~\ref{thm:PetrovskiiHs} provide
\begin{align*}
    \|\widetilde{V}_1-\widetilde{V}_{2}\|_{E_T^s}^2 \lesssim_{T,[M]_\alpha,\|M\|_{Y_T^{s+1}}} \int_0^T \|(M-\widetilde{M})\nabla (V_1-V_2)(t')\|_{\H^s(\T^d)}^2\,\dd t'.
\end{align*}
Owing to the product rule of in Proposition~\ref{prop:rhsSobo} we infer from the previous inequality the following 
\begin{align*}
  \|\widetilde{V}_1-\widetilde{V}_{2}\|_{E_T^s} &\lesssim_{T,[M]_\alpha,\|M\|_{Y_T^{s+1}}} \|M-\widetilde{M}\|_\infty \|V_1-V_2\|_{Y_T^{s+1}} + \|M-\widetilde{M}\|_{Y_T^{s+1}} \|V_1-V_2\|_{X_T^s}\\
  &\lesssim_{T,[M]_\alpha,\|M\|_{Y_T^{s+1}}} \ep \|V_1-V_2\|_{E_T^s} 
\end{align*}
Choosing $\ep$ small enough implies that the map $V\mapsto \widetilde{V}$ is a contraction from the Banach space $E_T^s$ to itself, thus $\mathscr{S}$ contains~$\widetilde{M}$. Finally, we proved that $\mathscr{S}$ is open and closed in~$\mathscr{C}^0(Q_T;\mathscr{P})$ which is connected; the set  $\mathscr{S}$ is non empty (constant matrices belong to~$\mathscr{S}$ thanks to Proposition~\ref{prop:systlin}) so  $\mathscr{S}=\mathscr{C}^{0,\alpha}(Q_T;\mathscr{P})$. 

\subsection{Propagation of regularity}
Let us prove the following result.
\begin{coro}\label{coro:thm2}
Let~$s>d/2$ and~$(V^0,F) \in \mathcal D_T^s$. Consider~$V \in E_T^s$ the associate solution to~{\rm(\ref{eq:Pt})} as given by Theorem~{\rm\ref{thm:exlin}}. Let~$s' \in [s,s+1]$ be given and assume furthermore that~$(V^0,F)$ belongs to~$\mathcal D_T^{s'}$. Then~$V$ actually belongs to $E_T^{s'}$ and satisfies
\begin{equation}\label{thm2Hsestimate:more}  
 \|  V \|_{E_T^{s'}}  \lesssim_{T,[M]_{\alpha},\|M\|_{Y_T^{s+1}}}  \|  (V ^0,F)\|_{\mathcal D_T^{s'}}   \, .\end{equation}
 \end{coro}
\begin{proof}[Proof]
Thanks to the well-posedness setting of Theorem~\ref{thm:exlin} we only need to prove the estimate with smooth $V$ and $F$.
We have thanks to estimate \eqref{thm2Hsestimate} $$
 \|  V \|_{E_T^s}   \lesssim_{T,[M]_{\alpha},\|M\|_{Y_T^{s+1}}} \|  (V ^0,F)\| _{\mathcal D_T^{s}} \, . $$
 Using the interpolation $\H^{s'}(\T^d) = [\H^{s}(\T^d),\H^{s+1}(\T^d)]_\theta$ for~$
 s<s'<s+1$, an interpolation argument shows that it is enough to prove~(\ref{thm2Hsestimate:more}) for~$s'=s+1$.
Now, for any spatial derivative~$\partial_\ell$, we note that $Z_\ell:=\partial_\ell V$ solves
\begin{align*}
    \partial_t Z_\ell - \sum_{k=1}^d  \partial_k \big[ M \partial_k Z_\ell\big] = \partial_\ell F + \sum_{k=1}^d  \partial_k \big[ \partial_\ell M \partial_k V\big].  
\end{align*}
Using the assumption on $M$ and that $\nabla V \in L^2(0,T;\H^s(\T^d))$, the previous equality already implies that $\partial_t Z_\ell \in\L^1(0,T;\H^s(\T^d))$ which implies $Z_\ell\in\mathscr{C}^0([0,T];\H^s(\T^d))$; this establishes  that~$V\in\mathscr{C}^0([0,T];\H^{s+1}(\T^d))$. Then, we use once more estimate \eqref{thm2Hsestimate}     for each $\ell$ to infer after summation using the algebra structure of $\H^s(\T^d)$, for $t\leq T$
$$ 
   \|  \nabla V(t) \|_{E_T^s} ^2    \lesssim_{T,[M]_{\alpha},\|M\|_{Y_T^{s+1}}} \|  (\nabla V ^0,\nabla F)\|_{\mathcal D_T^{s}}^2 + \int_0^T \big \|\nabla M(t)\|_{\H^{s }(\T^d)}^2\| {V}(t)\|_{\H^{s+1}(\T^d)}^2\,\dd t 
$$and the conclusion follows by Gronwall's inequality. $\qedhere$
\end{proof}

\section{Proof of Theorem~\ref{thm:exloc}
}\label{sec:road}

Now that we have a clear setting of well-posedness for the linear problem \eqref{eq:Pt}, in order to prove Theorem~\ref{thm:exloc} we aim to solve \eqref{eq:PtNL} for a given $(U^0,F)\in \mathcal D_\infty^s$ on a small interval $[0,T_0]$, by means of a Picard scheme. For $s>d/2$ and $\alpha_s\in(0,1)$ given by Lemma~\ref{lem:sobo2}, we will use again in this paragraph the space $G_T^s$ introduced in \eqref{def:GTs} and by a small abuse of notation we will write $G_T^s(U^0)$ for the (closed) affine subspace of $G_T^s$ constituted of those vector fields $U$ satisfying $U(0)=U^0$. Note that~$G_T^s(U^0)$ is a complete metric space.

\subsection{Existence and uniqueness in a small ball of $G_T^s(U^0)$}\label{subsec:exunismall}
Given $(U^0,F)\in \mathcal D_\infty^s$, we consider the following map
\begin{align*}
  \Theta : G_T^s &\longrightarrow G_T^s(U^0) \\
   U &\longmapsto U^\star,
\end{align*}
where $U^\star$ is the only element (existence and uniqueness stem from Theorem~\ref{thm:exlin}) of $G_T^s(U^0)$ solving $L_{A(U)} U^\star = F$. The use of Theorem~\ref{thm:exlin} is justified because $U\in G_T^s \Rightarrow A(U)\in G_T^s$, using Lemma~\ref{lem:lipsob} with $\Phi=A$, $f=U$ and $g=0$. Just as we did in the proof of Corollary~\ref{coro:useful} in Paragraph~\ref{subsec:useful} we recover in this way that $A(U)\in Y_T^{s+1} \cap \mathscr{C}^{0,\alpha_s}(Q_T)$. To see that~$A(U)$  belongs to~$ X_T^s$ we rely on Lemma~\ref{lem:lipsob} (with~$f = U$ and~$g = 0$) using $X_T^s\hookrightarrow\L^\infty(Q_T)$.

\medskip

Now that $\Theta$ is well-defined for all $T>0$, we hope to find a time small enough so as $\Theta$ becomes a contraction. Since \[L_{A(U_1)}(U_1^\star - U_2^\star) = \sum_{k=1}^d \partial_k\big[(A(U_1)-A(U_2))\partial_k U_2^\star\big],\]we infer from Corollary~\ref{coro:useful} 
\begin{align*}
\|U_1^\star - U_2^\star\|_{G_T^s} \lesssim_{T,\|U_1\|_{G_T^s}} \|A(U_1)-A(U_2)\|_{X_T^s} \|\nabla U_2^\star\|_{Y_T^s} \, .
\end{align*}
Thanks to  Lemma~\ref{lem:lipsob} we have
\begin{align}\label{ineq:lipAvrai}
    \forall U_1,U_2\in X_T^s,\qquad \|A(U_1)-A(U_2)\|_{X_T^s} \leq \ffi(\|U_1\|_{X_T^s}+\|U_2\|_{X_T^s})\|U_1-U_2\|_{X_T^s} \, ,
\end{align}
where $\ffi$ is some increasing function related to $A$, so
\begin{align}\label{ineq:contrapa}
\|U_1^\star - U_2^\star\|_{G_T^s} \lesssim_{T,\|U_1\|_{G_T^s} , \|U_2\|_{X_T^s}} \|U_1-U_2\|_{X_T^s} \|\nabla U_2^\star\|_{Y_T^s} \, .
\end{align}
  It seems clear, due to the presence of the multiplicative constant, that no global contraction rate can be achieved for $\Theta$ and we need to localize this map on some ball of $G_T^s$ to hope for a contraction. On the other hand $\|\nabla U_2^\star \|_{Y_T^s}^2$ will indeed tend to be small as $T\rightarrow 0$, but with a decay which will depend on $U_2$ and not only on the data of the problem. The strategy is thus to choose as fixed profile $U_F\in G_{T}^s(U^0)$ around which the fixed-point will be searched. More precisely, we have the following lemma, recalling here the notation
  \[ \|  (U ^0,F)\| _{\mathcal D_T^{s}}:=\|U^0\|_{\H^s(\T^d)}+\|F\|_{Y_T^{s-1}}+\int_0^T |\langle F(t)\rangle|\,\dd t\, ,\]
to keep track of the data's size.
\begin{lem}\label{lem:contraction}
                                 Fix $s>d/2$ and data~$(U^0,F) \in \mathcal D_\infty^s$. For any $T>0$ there exists a unique $U_F\in G_T^s(U^0)$ such that $L_{A(0)} U_F = F$ and it satisfies the following estimate 
                                 \begin{align}\label{ineq:UF}
\|U_F\|_{G_T^s} \lesssim \|  (U ^0,F)\| _{\mathcal D_T^{s}}\, ,
                                 \end{align}
                                 where $\lesssim$ depends only on the matrix $A(0)$. Furthermore, there exists an increasing function~$g$ such that, for any $r\in(0,1]$ and $T>0$, the closed ball $\B_{G_T^s}(U_F,r)$ is stabilized by $\Theta$ as soon as $T$ and $r$ satisfy
                                 \begin{align}\label{ineq:stabT}
                                   g(T+1+ \|  (U ^0,F)\| _{\mathcal D_T^{s} })\|\nabla U_F\|_{Y_T^s} \leq r\, .
                                 \end{align}
                                  Under this condition $\Theta$ is lipschitz on $\B_{G_T^s}(U_F,r)$ with a lipschitz constant bounded by
\[g(T+1+ \|  (U ^0,F)\| _{\mathcal D_T^{s}})(r+\|\nabla U_F\|_{Y_T^s})\, .\]
\end{lem}
                               \begin{proof}
                        We note that         $U_F$ is in fact nothing more than $\Theta(0)$, so its existence and uniqueness are not new since we already proved that $\Theta$ is well-defined. However, the symbol $\lesssim$ in estimate \eqref{ineq:UF} is independent of the time variable, and this is important. To obtain this, we rely on the setting for constant matrix fields given in Proposition~\ref{prop:systlin}, using Lemma~\ref{lem:sobo2} to add the Hölder norm in the estimate and adding the time evolution of the spatial average as we did in the beginning of Section~\ref{sec:exlin}.   This proves~(\ref{ineq:UF}).

                        For any $r\in(0,1]$, if $U_1 $ belongs to~$\B_{G_T^s}(U_F,r)$, we infer from \eqref{ineq:contrapa} applied with~$U_2=0$ and~$U_2^\star = \Theta(0)=U_F$ that
\begin{align}\label{ineq:stabBf}
\|U_1^\star - U_F\|_{G_T^s} \lesssim_{T,\|U_1\|_{G_T^s} ,1} \|U_1\|_{X_T^s}\|\nabla U_F\|_{Y_T^s} \, .
\end{align}
 Using $\|U_1-U_F\|_{G_T^s}\leq r \leq 1$ together with \eqref{ineq:UF}, we get, for some increasing function $g_1$,
\begin{align*}
\|U_1^\star - U_F\|_{G_T^s} \leq  g_1(T+1+ \|  (U ^0,F)\| _{\mathcal D_T^{s}})\|\nabla U_F\|_{Y_T^s}\, .
\end{align*}
Now if indeed $T$ is small enough so as \[g_1(T+1+ \|  (U ^0,F)\| _{\mathcal D_T^{s}}) \|\nabla U_F\|_{Y_T^s} \leq r\, ,\]
we have that $\Theta(U_1) = U_1^\star$ lies in $\B_{G_T^s}(U_F,r)$ so that this closed ball is indeed preserved by~$\Theta$. Finally to evaluate the lipschitz constant of $\Theta$ on this ball, we use once more \eqref{ineq:contrapa} with $U_1,U_2\in \B_{G_T^s}(U_F,r)$ and the triangular inequality to infer, for some increasing function $g_2$ 
\begin{align*}
  \|U_1^\star - U_2^\star\|_{G_T^s} &\leq g_2(T+2\|U_F\|_{G_T^s}+2r)  \|\nabla U_2\|_{Y_T^s}\|U_1-U_2\|_{G_T^s}\\
  &\le g_3(T+1+ \|  (U ^0,F)\| _{\mathcal D_T^{s}})(r+\|\nabla U_F\|_{Y_T^s}) \|U_1-U_2\|_{G_T^s} \,,
\end{align*}
where we used \eqref{ineq:UF}, $r\leq 1$ and $g_3(z): = g_2(2z)$. The proof follows for $g:=\max(g_1,g_3)$.
\end{proof}
The proof of Theorem~\ref{thm:exloc} will now follow from Lemma~\ref{lem:contraction} and Picard's fixed-point theorem. For any $T>0$, if we choose
\[r=r_T := g(T+1+ \|  (U ^0,F)\| _{\mathcal D_T^{s}})\|\nabla U_F\|_{Y_T^s} \,,\]
where $g$ is the increasing function of Lemma~\ref{lem:contraction}, estimate \eqref{ineq:stabBf} is automatically satisfied and $\B_{G_T^s}(U_F,r_T)$ is thus preserved by $\Theta$. For this choice $r=r_T$, the bound of the lipschitz constant given in Lemma~\ref{lem:contraction} becomes strictly less than $1$ as soon as
\begin{align*}
r_T + \|\nabla U_F\|_{Y_T^s} = \big[1+g(T+1+ \|  (U ^0,F)\| _{\mathcal D_T^{s}})\big]\|\nabla U_F\|_{Y_T^s} < \frac{1}{g(T+1+ \|  (U ^0,F)\| _{\mathcal D_T^{s}})} \,,
\end{align*}
which ultimately takes the form
\begin{align}\label{ineq:ultimate}
\|\nabla U_F\|_{Y_T^s} < \frac{1}{h(T+1+ \|  (U ^0,F)\| _{\mathcal D_T^{s}})} \,,
\end{align}
for yet another increasing function $h$. This ends the proof of local existence for Theorem~\ref{thm:exloc} because  \[\|\nabla U_F\|_{Y_T^s}^2 := \int_0^T \|\nabla U_F(t)\|_{\H^s(\T^d)}^2\,\dd t\]
tends to $0$ as $T\rightarrow 0$, so we recover indeed for $T$ small enough that $\Theta$ induces a contraction map on $\B_{G_T^s}(U_F,r_T)$ and have thus a fixed-point.
\subsection{Global uniqueness and stability} 
In the previous paragraph, we have shown the existence of a solution on some small time interval. We have also, by construction, proved its uniqueness but only in an appropriate neighboorhood of $U_F$. In this short paragraph, we establish global uniqueness (as stated in Theorem~\ref{thm:exloc}) of this solution in $E_T^s$ by means of the stability estimate \eqref{eq:stabNL} (which obviously implies uniqueness). To prove this estimate, we rely once more on Theorem~\ref{thm:PetrovskiiHs}.

\medskip

We consider therefore~$U_1$ and~$U_2$ two solutions, associated with   data~$(U_1^0,F_1)$ and~$(U_2^0,F_2)$ respectively, and let~$T>0$ be a common time of existence; both solutions are in $E_T^s$ by assumption and since their time derivatives belong to $Y_T^{s-1}$, we have of course that both of them are in $G_T^s$ (see Lemma~\ref{lem:sobo2}), just as the solutions we built above.

\medskip

We set~$V:=U_1-U_2$ and notice that
\[L_{A(U_1) }V=\sum_{k=1}^d\partial_k\big([A(U_1)-A(U_2)]\partial_k U_2\big) + F_1-F_2\, .\]
Since $L_{A(U_1) }V = \mathbb{P}_0\,L_{A(U_1) }V = L_{A(U_1)}\mathbb{P}_0 V$, we infer from Corollary~\ref{coro:useful}, using that $\H^s(\T^d)$ is an algebra,
$$\begin{aligned}
    \|\mathbb{P}_0 V \|_{G_T^s}^2& \lesssim_{T, \|U_1\|_{G_T^s}} \|\mathbb{P}_0 V(0)\|_{\H^s(\T^d)}^2\\
    & \quad + \int_0^T \|A(U_1)(t)-A(U_2)(t)\|_{\H^s(\T^d)}^2\|\nabla U_2(t)\|_{\H^{s}(\T^d)}^2 \,\dd t \\
&\qquad +  \int_0^T \|{\mathbb P}_0(F_1-F_2)  (t)\|_{\H^{s-1}(\T^d)}^2 \, \dd t\, .
\end{aligned}
$$
 Using once more the lipschitz estimate of Lemma~\ref{lem:lipsob} with $\Phi=A$, $f=U_1$ and $g=U_2$, we infer as $G_T^s\hookrightarrow X_T^s \hookrightarrow \L^\infty(Q_T)$, after adding $\langle V(t)\rangle^2 = \Big(\langle V(0)\rangle+  \int_0^t \langle F_1-F_2 \rangle (t') \, \dd t'\Big)^2$ on both sides   and replacing $T$ by an arbitrary $t\in[0,T]$
 $$
\begin{aligned}
    \|V(t)\|_{\H^s(\T^d)}^2 &+ \int_0^t \|\nabla V(t')\|_{\H^s(\T^d)}^2 \,\dd t'\\ 
    &\quad \lesssim_{T,\|U_1\|_{G_T^s},\|U_2\|_{G_T^s}} \|V(0)\|_{\H^s(\T^d)}^2+ \int_0^t \|V(t')\|_{\H^s(\T^d)}^2\|\nabla U_2(t')\|_{\H^{s}(\T^d)}^2 \,\dd t \\
  &\qquad  + \int_0^t \|{\mathbb P}_0(F_1-F_2)  (t')\|_{\H^{s-1}(\T^d)}^2 \, \dd t ' +\Big(  \int_0^t \langle F_1-F_2 \rangle (t') \, \dd t'\Big)^2 \, .
\end{aligned}
$$
 Grönwall's lemma allows to conclude and establish \eqref{eq:stabNL}. Theorem~\ref{thm:exloc}
 is proved. \qed

\section{Proof of Theorem~\ref{thm:lbu}}\label{prooflbu}

In this last section, we prove Theorem~\ref{thm:lbu}. The three coming paragraphs respectively focus on points $(i)$, $(ii)$ and $(iii)$ in the statement of the theorem.
\subsection{Global solutions for small data}
We rely, just as we did in Subsection~\ref{subsec:exunismall}, on a Picard scheme. We use the same map $\Theta:U\mapsto U^\star$ introduced at the beginning of Subsection~\ref{subsec:exunismall} and defined on $G_T^s(U^0)$. Instead of $\Theta(0) = U_F$, we shall choose $0$ as center of the ball. We first note for any $U\in G_T^s(U^0)$ that
\begin{align}\label{eq:globT}
L_{A(0)}U^\star = \sum_{k=1}^d\partial_k\big([A(U)-A(0)]\partial_k U^\star\big) + F\,.
\end{align}
Now recall the existence of an increasing function $\ffi$ depending only on $A$ and satisfying~\eqref{ineq:lipAvrai}. Together with the algebra structure of $\H^s(\T^d)$, we then   write
 \begin{align*}\Big\|\big[A(U)-A(0)\big]\partial_k U^\star\Big\|_{Y_T^s}&\leq \|A(U)-A(0)\|_{X_T^s}\|\nabla U^\star\|_{Y_T^s} \\
&\leq \ffi(\|U\|_{X_T^s}) \|U\|_{X_T^s}\|U^\star \|_{Y_T^{s+1}}\\
&\leq \ffi(\|U\|_{G_T^s}) \|U\|_{G_T^s}\|U^\star \|_{G_T^s}\,.\end{align*}
Then, returning to \eqref{eq:globT}, the point is, instead of using Corollary~\ref{coro:useful}, to rely on Proposition~\ref{prop:systlin}, for which the estimate is independent of the time variable. More precisely, just as we did for $U_F$ in the proof of Lemma~\ref{lem:contraction}, using Lemma~\ref{lem:sobo2} to add the Hölder norm in the estimate and adding the time evolution of the spatial average, we infer
\begin{align}\label{ineq:globT}
\|U^\star\|_{G_T^s} \leq \textnormal{C}_{A(0)} \Big( \|  (U ^0,F)\| _{\mathcal D_\infty^{s}} + \ffi(\|U\|_{G_T^s})\|U\|_{G_T^s} \|  U^\star\|_{G_T^s}\Big) \,,
\end{align}
where $\textnormal{C}_{A(0)}$ depends only on the matrix $A(0)$ and $ \|  (U ^0,F)\| _{\mathcal D_\infty^{s}}$ is finite by assumption. Now, fix $r\in(0,1]$ such that $r\textnormal{C}_{A(0)}\ffi(1)<1$. For any $U\in \B_{G_T^s}(0,r)$ estimate \eqref{ineq:globT} implies
\begin{align*}
\|U^\star\|_{G_T^s} \leq \frac{\textnormal{C}_{A(0)}}{1-r\ffi(1)\textnormal{C}_{A(0)}}  \|  (U ^0,F)\| _{\mathcal D_\infty^{s}}\,.
\end{align*}
In particular, for any $r$ as above, if
\begin{align}\label{ineq:smallnessdd}
   \|  (U ^0,F)\| _{\mathcal D_\infty^{s}} \leq r \left(\frac{1}{\textnormal{C}_{A(0)}}-r\ffi(1)\right)\,,
  \end{align}
we have just proved that for all times $T$ the closed ball $\B_{G_T^s}(0,r)$ is preserved by $\Theta$. What about the lipschitz constant of $\Theta$ in that ball ? Just as in \eqref{eq:globT} we rely on the flow of the constant matrix field $A(0)$ writing for $U_1,U_2$ in $\B_{G_T^s}(0,\textnormal{C}_{A(0)}/2)$ 
$$
\begin{aligned}
L_{A(0)}(U_1^\star-U_2^\star)& = \sum_{k=1}^d \partial_k\Big(\big[A(U_1)-A(0)\big]\partial_k \big[U_1^\star-U_2^\star\big]\Big) \\
& \qquad + \sum_{k=1}^d \partial_k\Big(\big[A(U_1)-A(U_2)\big]\partial_k U_2^\star\Big)\,.
\end{aligned}
$$
We use as above Proposition~\ref{prop:systlin} together with Lemma~\ref{lem:sobo2} to estimate the full $G_T^s$ norm, and \eqref{ineq:lipAvrai} just as above with the pair $(U_1,0)$ and $(U_1,U_2)$: since~$U_1^\star-U_2^\star$ vanishes at the initial time, we obtain  that for the same increasing function $\ffi$ as before 
$$
\begin{aligned}
\|U_1^\star-U_2^\star\|_{G_T^s}& \leq \textnormal{C}_{A(0)} \Big(\ffi(\|U_1\|_{X_T^s})\|U_1\|_{X_T^s}\|U_1^\star-U_2^\star\|_{Y_T^{s+1}} \\
 &\qquad  + \ffi(\|U_1\|_{X_T^s}+\|U_2\|_{X_T^s})\|U_1-U_2\|_{X_T^s}\|U_2^\star\|_{Y_T^{s+1}}\Big)\,.
\end{aligned}
$$
Since $U_1$, $U_2$ and $U_2^\star$ both belong to $\B_{G_T^s}(0,r)$ with $G_T^s\hookrightarrow X_T^s\cap Y_T^{s+1}$ (with  operator norm less than $1$), and $r\in(0,1]$ is such that $r\textnormal{C}_{A(0)}\ffi(1)<1$, we get
\begin{align*}
\|U_1^\star-U_2^\star\|_{G_T^s} \leq r\frac{\textnormal{C}_{A(0)}\ffi(2)}{1-r\textnormal{C}_A(0)\ffi(1)} \|U_1-U_2\|_{G_T^s} \, .
\end{align*}
So we first choose $r\in(0,1\wedge(\textnormal{C}_{A(0)}\ffi(1))^{-1})$ small enough so as
\[r\frac{\textnormal{C}_{A(0)}\ffi(2)}{1-r\textnormal{C}_A(0)\ffi(1)}<1\,,\]
and this defines the threshold \eqref{ineq:smallnessdd} for $ \|  (U ^0,F)\| _{\mathcal D_\infty^{s}}$ below which we have a solution for all times, thanks to Picard's fixed-point theorem. Point~(i) of  Theorem~\ref{thm:lbu} is proved.

\subsection{Finer description of the lifetime}
Let us prove point~(ii) of  Theorem~\ref{thm:lbu}. We consider~$(U^0,F)$ in~$\mathcal D_\infty^{s}$. We recall the sufficient condition \eqref{ineq:ultimate}  for a solution to 
exist  in~$E_T^s$
(where~$h$ is some increasing function). For any $T>0$ satisfying this condition, we have~$T_s^\star \geq T$. 
Now consider~$\sigma $ any real number in~$ (d/2,s)$ such that~$s \leq \sigma+1$.  Using the interpolation $\H^\sigma(\T^d) = [\H^{\sigma-1}(\T^d),\H^s(\T^d)]_\theta$ we can write for some $\theta\in(0,1)$
\begin{align*}
  \|\nabla U_F(t)\|_{\H^{\sigma}(\T^d)} &\leq \|\nabla U_F(t)\|_{\H^{\sigma-1}(\T^d)}^\theta \|\nabla U_F(t)\|_{\H^s(\T^d)}^{1-\theta}\\
                                        &\leq  \|U_F(t)\|_{\H^{\sigma}(\T^d)}^\theta \|\nabla U_F(t)\|_{\H^s(\T^d)}^{1-\theta}.
\end{align*}
We have thus
\begin{align*}
  \|\nabla U_F\|_{Y_T^\sigma} &\leq \|U_F\|_{X_T^\sigma}^\theta \|\nabla U_F\|_{Y_T^s}^{1-\theta} T^{\theta/2},\\
                                        &\leq    \|  (U ^0,F)\| _{\mathcal D_T^{s}} T^{\theta/2},
\end{align*}
where we used $\sigma\leq s$ and \eqref{ineq:UF}. Now let us explore   the sufficient condition \eqref{ineq:ultimate}   in the~$\H^\sigma(\T^d)$ setting:  we see that it is satisfied as soon as
\begin{align*}
T^{\theta/2} \leq \frac{1}{ \|  (U ^0,F)\| _{\mathcal D_T^{s}}}\frac{1}{h(T+1+ \|  (U ^0,F)\| _{\mathcal D_T^{\sigma}})} \, \cdotp
  \end{align*}
  Using again $ \|  (U ^0,F)\| _{\mathcal D_T^{\sigma}}\leq  \|  (U ^0,F)\| _{\mathcal D_T^{s}}$, the previous inequality is satisfied as soon as
\begin{align*}
T \leq \Phi(T+1+ \|  (U ^0,F)\| _{\mathcal D_T^{s}})\, ,
\end{align*}
where $\Phi(z)=z^{-2/\theta}h(z)^{-2/\theta}$. Since $\Phi$ is decreasing, it is also the case of the function~$\Phi^{-1}$ and~$\Psi:z\mapsto \Phi^{-1}(z)-z-1$, and we find that
$$
T_\sigma^\star  (U ^0,F)\geq  \varphi (  \|  (U ^0,F)\| _{\mathcal D_T^{s}})
$$
 where $\ffi:=\Psi^{-1}$ is indeed decreasing. It now suffices to prove that~$T_s^\star  (U ^0,F)\geq T_\sigma^\star  (U ^0,F)$ (note that the reverse inequality is obvious). But this is actually an immediate consequence
 of the propagation of regularity result stated in Corollary~\ref{coro:thm2}, so point~(ii) of  Theorem~\ref{thm:lbu} is proved.

\subsection{Blow-up for finite lifetime}\
This follows directly from~(ii).

\section{Sign preservation}\label{proofprop:sign}
In this section we prove   Proposition \ref{prop:sign} and Theorem \ref{thm:SKT},   which in particular leads to a     well-posedness result of the SKT system as explained in the introduction of this paper.
\subsection{Proof of Proposition \ref{prop:sign}}
Let us  consider~$A$ a smooth sign-preserving matrix field in the sense of Definition~\ref{def:nonneg}, and~$U$ a smooth solution to \eqref{eq:PtNL:reac} (namely in $E^s_T$ for~$s>d/2+2$). We assume that the data~$U^0\in\H^s(\T^d)$ and~$F$ in~$Y_\infty^{s-1}$ are non-negative. Consider the set $\{t \in[0,T] \,:\, U(s)\geq 0, \forall s\in[0,t]\}$ which is non-empty (it contains $0$ by assumption), and let's assume that its supremum $t_\star$ is strictly less than $T$. By the sign-preserving property, we know that $A(U) = D(U)+ \textnormal{diag}(U)B(U)$, where $D(U)\geq \alpha \textnormal{I}_N$ whenever $U\geq 0$. In particular, by a standard continuity argument we have the existence of $t^\star\in (t_\star,T)$ such that $D(U)\geq 0$ on $[0,t^\star]$.

  \medskip
  
  Since $R(U)= \textnormal{diag}(U)\rho(U)$, we can write for all $i\in\llbracket 1,N\rrbracket$, with obvious notations 
  \begin{align*}
    \partial_t u_i - \sum_{j=1}^d \partial_k \big[d_i(U)\partial_k u_i\big]-\sum_{k=1}^d \sum_{j=1}^N \partial_k \big[u_i b_{ij}(U) \partial_k u_j \big] = f_i + u_i \rho(U)\, .
  \end{align*}
For a real function $f$, we note $f^- = - f\mathbf{1}_{f<0}$ its negative part and recall that whenever $f$ has $\W^{1,1}(\T^d)$ regularity, the formula holds~$\nabla f^- = -\mathbf{1}_{f<0} \nabla f$. Multiplying the previous equation by $-u_i^-$ we infer after integration on $[0,t]\times\T^d$ for $t\in[0,t^\star]$ 
  \begin{multline*}
\frac12 \|u_i^-(t)\|_2^2  + \int_0^t \int_{\T^d} d_i(U)|\nabla u_i^-|^2 \, \dd x \dd s= \frac12 \|u_i^-(0)\|_2^2 - \int_0^t \int_{\T^d} u_i^- f_i \, \dd x \dd s \\+ \int_0^t \int_{\T^d} (u_i^-)^2 \rho(U) \, \dd x \dd s- \sum_{k=1}^d \sum_{j=1}^N \int_0^t \int_{\T^d} u_i^- \partial_k u_i^- b_{ij}(U) \partial_k u_j\, \dd x \dd s\, .
  \end{multline*}
  Using $U^0=U(0)\geq 0$, $F\geq0$, $\rho(U)\in\L^\infty(Q_{T})$ and $d_i(U)\geq 0$ (because $t\in[0,t^\star]$), we get by another integration by parts to handle the last term
$$  \begin{aligned}
    \frac12 \|u_i^-(t)\|_2^2 & \leq \|\rho(U)\|_{\L^\infty(Q_T)} \int_0^t \|u_i^-(s)\|_2^2 \,\dd s \\
    & \quad + \frac12 \sum_{k=1}^d \sum_{j=1}^N \int_0^t \int_{\T^d} (u_i^-)^2   \partial_k\big[b_{ij}(U) \partial_k u_j\big]
    \, \dd x \dd s\, .
  \end{aligned}
$$
 The fact that~$s>2+d/2$ is enough to justify all the previous computations and claim furthermore that $\partial_k \big[b_{ij}(U)\partial_k u_j\big]$ belongs to $\L^\infty(Q_T)$. This leads eventually to an estimate of the form
  \begin{align*}
 \|u_i^-(t)\|_2^2 \lesssim_{\|U\|_{E^s_T}} \int_0^t \|u_i^-(s)\|_2^2\,\dd s\, ,
  \end{align*}
and Grönwall's inequality leads to the fact that~$u_i^-=0$ on $[0,t^\star]$ which is in clear contradiction with the definition of $t_\star$, so $t_\star = T$. Proposition \ref{prop:sign} is proved. \qed

 \subsection{Proof of Theorem \ref{thm:SKT}}

 The proof consists in transforming system~(\ref{eq:PtNL:reac}) into one for which one can apply Proposition \ref{prop:sign}. 

  \medskip
  
  First let us   check that one can assume without loss of generality that~$s>d/2+2$. A   consequence of Corollary~\ref{coro:reac} is indeed the following. Pick~$U^0\in\H^s(\T^d)$ and~$F$   in~$Y_\infty^{s-1}$ with~$s>d/2$, and consider smooth approximations $U^0_\ep$ and $F_\ep$ of $U^0$ and $F$ respectively in~$\H^s(\T^d)$ and $Y_\infty^{s-1}$. Fix any time~$T<T_s^\star(U^0,F)$. Combining estimate \eqref{eq:stabNL} of Theorem~\ref{thm:exloc} and point $(iii)$ of Theorem~\ref{thm:lbu}, we have for $\ep$ small enough $T_s^\star(U^0_\ep,F_\ep) \geq T$ and $(U_\ep)_\ep\rightarrow  U$ in $E_T^s$ as~$\ep$ goes to zero. As the previous convergence preserves non-negativeness, we can therefore assume without loss of generality that $s$ is as large as needed. 
 
 \medskip
 
Now, Theorem~\ref{thm:SKT} follows from Proposition~\ref{prop:quadrant} of Appendix Section~\ref{app:retract}. Indeed, since~$\mathscr{P}$ is open (see Lemma~\ref{lem:gampet}) that's also the case of $\Omega:= A^{-1}(\mathscr{P})$ and by assumption $\Omega$ contains~$\R_{\geq 0}^N$. Thanks to Proposition~\ref{prop:quadrant} we have therefore a smooth function~$h$ sending $\R^N$ on $\Omega$ and leaving all points of $\R_{\geq 0}^N$ unchanged. Corollary~\ref{coro:reac} applies to find a solution~$U$ to the system~(\ref{eq:PtNL:reac})  where~$A$ is replaced by~$A\circ h$. Moreover since~$A\circ h$ is sign-preserving in the sense of Definition~\ref{def:nonneg}, Proposition~\ref{prop:sign}  shows that $U\geq 0$ on its lifetime so that~$A\circ h(U) = A(U)$ and we have built a non-negative solution to the original problem \eqref{eq:PtNL:reac}. Thanks to the uniqueness offered by our setting, this construction (and the corresponding maximal lifetime)  is independent of the map $h$ that we choose to define the solution.  
\qed

 \section{The end-point Besov case}
 \label{end-point Besov}
 As pointed out in the introduction,  we wish to achieve the critical setting, where   no assumption is made on the H\"older regularity of the initial data. To achieve this aim  we shall resort to Besov spaces, the definition of which may be found in Appendix~\ref{sec:LP}.

We define, for any given~$T>0$  and~$p \in [1,\infty)$, the analogue of the solution space~$E_T^s$ in the Besov setting:
$$
\E^s_T:=  \mathcal C^0 ([0,T];B^s_{p,1}) \cap \L^1 (0,T;B^{s+2}_{p,1})
$$
and the  analogue of the exterior force space~$Y_T^s$:
$$
{\mathbb Y}_T^{s}:=  \L^1 (0,T;B^{s }_{p,1})\, .
$$
For any matrix-valued function~$M$ on~$Q_T$ we define the quantity
 \begin{equation}\label{defnormalphabesov}
[M]_{0} :=\|M\|_{\L^\infty(Q_T)}+\eta(M)^{-1} \, ,\end{equation}
and~$ \omega_M  $   the modulus of continuity of~$M$ on~$Q_T$:
$$
\omega_M(r):=\sup\Big\{ \vertiii {M(z_1)- M(z_2)}\, , \quad z_i \in Q_T \, , \quad |z_1-z_2| \leq r\Big\}\, .
$$
Now let us set  data space~$
   \mathbb D_T^\frac dp:=B^\frac dp_{p,1}\times{\mathbb Y}_T^\frac dp      $. Notice that in particular the initial data~$U^0$ is continuous by the embedding of the set~$B^\frac dp_{p,1}(\T^d)$ into~$ \mathcal C^0(\T^d)$, but no more H\"older regularity holds in general. The size of the data is  measured by
         $$
\|(U^0,F)\|_{   \mathbb D_T^\frac dp}:= \|U^0\|_{B^\frac dp_{p,1}} + \|F\|_{{\mathbb Y}_T^\frac dp   } + \int_0^T \left|\langle F(t)\rangle\right|\,\dd t\, .
$$
We shall prove the following result.  Note that  for simplicity we do not prove in this setting all the results obtained in the previous Sobolev setting, but it should be clear from the proof of the result below that there would be no difficulty in doing so.
      \begin{thm}[Local well-posedness] \label{thm:exlocbesov}
   Consider a smooth~$A:\R^N\rightarrow \mathscr{P}$. For any~$(U^0,F)$ belonging to~$ \mathbb D_\infty^\frac dp$ for some~$p \in [1,\infty)$ there exists  $T>0$ and a unique element $U$ of $\E^\frac dp_T$ which solves the parabolic Cauchy problem  \eqref{eq:PtNL} on $Q_T$. \end{thm}
  The proof of that result relies on the following linear estimate, which will   be used to implement a fixed point argument.
    \begin{thm} \label{thm:PetrovskiiBesov}
    Let~$T>0$   and consider a matrix field $M\in\mathscr{C}^{0}(Q_T;\mathscr{P})\cap \E^\frac dp_T$. For any $V$  such that $V(0)\in B^\frac dp_{p,1}$, $L_{M }V\in{\mathbb Y}_T^\frac dp  $ and $\langle V(t)\rangle=0$ for all $t\in[0,T]$,  one has   \begin{align}\label{thm2Hsestimate:bes}  
\|  V \|_{ \E^\frac dp_T}  \lesssim_{T,[M]_{0},\omega_M,\|M\|_{ \E^\frac dp_T}}  \|V^0\|_{B^\frac dp_{p,1}} +  \|L_{M }V\|_{{\mathbb Y}_T^\frac dp}\, .
\end{align}
\end{thm}
  In the above lemma, $  \lesssim_{T,\omega_M,[M]_0,,\|M\|_{ \E^\frac dp_T}} $ means a multiplicative constant which is an increasing function of~$T+\omega_M+[M]_0+\|M\|_{ \E^\frac dp_T}$.
 \subsection{The constant coefficient case in $B^s_{p,1}$}
 Let us prove the following result.
  \begin{prop}\label{prop:systlinbesov}
  Let~$M\in\mathscr{P}_\delta$ for some~$\delta>0$. Consider any~$s \in \R$ and~$p \in [1,\infty)$, and fix~$V^0\in B^s_{p,1}$, $F\in {\mathbb Y}_T^{s}$. The Cauchy problem
 $$  \left\{
  \begin{aligned}
    \partial_t V-M \Delta V &= F,\\
   V_{|t=0} &= V^0,
  \end{aligned}
  \right.                                
$$
is well posed in~$B^s_{p,1}$, and moreover for all~$t \geq 0$ there holds for all~$T>0$
$$
\|     V  \|_{\L^\infty([0,T];B^s_{p,1})} +  \|   V \|_{\L^1([0,T];B^{s+2}_{p,1})} \leq  \Big(1+ \frac{ \vertiii{M}}{\delta}\Big)^N\Big( \|   V (0)\|_{B^s_{p,1}} +     \|   F \|_{\L^1([0,T];B^s_{p,1})}  \Big) \, .
$$ 
\end{prop}
\begin{proof}
We shall only prove the a priori estimate, and leave the well-posedness result to the reader. Similarly to the proof of Proposition~\ref{prop:systlin}, we use the Schur decomposition to write~$T:= \mathcal U M \mathcal U^\star$ with~$  T$ upper triangular and~$\mathcal U$ unitary. 
Let us set~$\tilde V:= \mathcal U^\star V$ and~$\tilde F:= \mathcal U^\star F$. Then applying the Littlewood-Paley operator~$\Delta_j $ to the equation implies that
$$
   \partial_t \Delta_j \tilde V-T \Delta\Delta_j \tilde V  = \Delta_j\tilde F \, ,
$$
and since~$T$ is upper triangular in particular the last component satisfies 
$$
   \partial_t \Delta_j \tilde V_N -d_N  \Delta\Delta_j \tilde V_N  = \Delta_j\tilde F_N  \, .
$$
Then we write
$$
 \Delta_j \tilde V_N(t) = e^{d_Nt \Delta}\Delta_j \tilde V_N(0) + \int_0^t e^{d_N(t-t') \Delta}\Delta_j \tilde F_N(t') \, \dd t' \, ,
$$
and we can use~(\ref{eq:heat}) to infer that
$$
\| \Delta_j \tilde V_N(t)\|_p \leq  e^{- cd_N2^{2j}t  }\| \Delta_j \tilde V_N(0)\|_p  +   \int_0^t  e^{- cd_N2^{2j}(t-t')  }\| \Delta_j\tilde F_N(t')\|_p \,\dd t' \, .
$$
Taking the sup norm in~$t \in [0,T]$ and using Young's inequality in time for the second term on the right-hand side provides
$$
 \| \Delta_j \tilde V_N \|_{\L^\infty([0,T];\L^p)} \leq \| \Delta_j \tilde V_N(0)\|_p  + \|\Delta_j\tilde F_N \|_{\L^1([0,T];\L^p)}
$$
while taking the~$L^1$ norm in time gives
$$
 \| \Delta_j \tilde V_N \|_{\L^1([0,T];\L^p)} \leq \frac {2^{-2j}}{c d_N} \Big(\| \Delta_j \tilde V_N(0)\|_p  + \|\Delta_j\tilde F_N \|_{\L^1([0,T];\L^p)}\Big)\, .
$$
Finally multiplying both inequalities by~$2^{js} $ and summing over~$j$ gives
\begin{equation}
\label{eq : VN}
\|    \tilde V_N  \|_{\L^\infty([0,T];B^s_{p,1})} +  \|    \tilde V_N \|_{\L^1([0,T];B^{s+2}_{p,1})} \leq \frac C {d_N} \Big( \|   \tilde V_N (0)\|_{B^s_{p,1}} +     \|    \tilde F_N \|_{\L^1([0,T];B^s_{p,1})}  \Big) \, .
\end{equation}
Now  we argue by iteration: there holds
$$
   \partial_t \Delta_j \tilde V_{N-1} -  d_{N-1}  \Delta\Delta_j \tilde V_{N-1} = \Delta_j\tilde F_{N-1} + r_{N-1}\Delta \Delta_j \tilde V_{N }  $$
   so as above and using~(\ref{eq:bernstein}) to bound~$\|\Delta \Delta_j \tilde V_{N } \|_p$ by~$C 2^{2j}\|  \Delta_j \tilde V_{N } \|_p$, we get
   $$
\begin{aligned}
\| \Delta_j \tilde V_{N-1}(t)\|_p  \lesssim  e^{- cd_{N-1}2^{2j}t  }\| \Delta_j \tilde V_{N-1}(0)\|_p & +   \int_0^t  e^{- cd_{N-1}2^{2j}(t-t')  }\| \Delta_j\tilde F_{N-1}(t')\|_p \, \dd t'  \\
&   +   \int_0^t  e^{- cd_{N-1}2^{2j}(t-t')  } 2^{2j}\|  \Delta_j \tilde V_{N } (t')\|_p \,\dd t'  \, .
\end{aligned}
$$
Again using twice Young's inequality  to deal with the time integrals, we find
$$
\begin{aligned}
\|    \tilde V_{N-1 } \|_{\L^\infty([0,T];B^s_{p,1})} +  \|    \tilde V_{N-1 }\|_{\L^1([0,T];B^{s+2}_{p,1})} &\leq \frac C {d_{N-1}} \Big( \|   \tilde V_N (0)\|_{B^s_{p,1}} +     \|    \tilde F_N \|_{\L^1([0,T];B^s_{p,1})}  \\
& \quad  +   \|  \tilde V_{N }  \|_{\L^1([0,T];B^{s+2}_{p,1})} \Big) \, .
\end{aligned}
$$
Plugging~(\ref{eq : VN}) into this inequality provides
$$
\begin{aligned}
\|    \tilde V_{N-1 } \|_{\L^\infty([0,T];B^s_{p,1})} +  \|    \tilde V_{N-1 }\|_{\L^1([0,T];B^{s+2}_{p,1})}& \leq \frac C {d_{N-1}} \Big( \|   \tilde V_N (0)\|_{B^s_{p,1}} +     \|    \tilde F_N \|_{\L^1([0,T];B^s_{p,1})}  \\
&  +  \frac C {d_N} \big( \|   \tilde V_N (0)\|_{B^s_{p,1}} +     \|    \tilde F_N \|_{\L^1([0,T];B^s_{p,1})}  \big)  \Big) \, .
\end{aligned}
$$
The proposition follows by iterating the argument. 
\end{proof}
  \subsection{The time-dependent   case in $B^s_{p,1}$}
 This
case is dealt with exactly as in the Sobolev case thanks to Proposition~\ref{prop:systlinbesov}.
\subsection{The variable coefficient  case in~$B^\frac dp_{p,1}$}  
\subsubsection{Reduction   to a single lemma}\label{subsec:redlipbesov}
In this subsection we explain how Theorem~\ref{thm:PetrovskiiBesov} can be recovered by the following  weaker result, by adapting the reasoning carried out in the Sobolev case above.
\begin{lem}\label{lem:goalL2:weakbesov}
 For any   map $M\in \mathscr{C}^{0}(Q_T;\mathscr{P})\cap L^\infty ([0,T];B^{\frac dp+2}_{p,1}) $ and any $V\in {\mathbb E}_T^\frac dp$ such that  $L_M V\in {\mathbb Y}_T^\frac dp$ and $\langle V(t)\rangle = 0$ for all $t\in[0,T]$, one has 
\begin{align*}
  \|V\|_{ {\mathbb E}_T^\frac dp}  \lesssim_{T,\omega_M,[M]_0}  \|V(0)\|_{B^\frac dp_{p,1}} + \| L_M V \|_{{\mathbb Y}_T^\frac dp }  + \Big(1+\|M\|^2_{ L^\infty ([0,T];B^{\frac dp+2}_{p,1})  }\Big)  \|V\|_ {{\mathbb Y}_T^\frac dp}   \, .
\end{align*}
 \end{lem}
Admitting for the moment the previous lemma, Theorem~\ref{thm:PetrovskiiBesov}
 can be proved thanks to an approximation argument. If $M\in\mathscr{C}^{0}(Q_T;\mathscr{P})\cap {\mathbb E}^\frac dp_T $, usual convolution properties lead to the existence of   matrix-valued functions $(M_\ep)_\ep$  for which
\begin{align}
  \label{ineq:Mep1besov}\|M_\ep\|_{\mathscr{C}^{0}(Q_T)}&\leq \|M\|_{\mathscr{C}^{0}(Q_T)}\, ,\\
  \label{ineq:Mep11besov}\|M_\ep\|_{{\mathbb Y}_T^{\frac dp+2}}&\leq \|M\|_{{\mathbb Y}_T^{\frac dp+2}}\, ,\\
  \label{ineq:Mep2besov}\lim_{\ep \to 0}\|M-M_\ep\|_\infty &= 0 \, ,\\
  \label{ineq:Mep3besov}\|M_\ep\|_{ L^\infty ([0,T];B^{\frac dp+2}_{p,1})  }&\leq (1+\ep^{-2})\|M\|_{ L^\infty ([0,T];B^{\frac dp}_{p,1})  }\, .
 \end{align}
 Because of \eqref{ineq:Mep2besov} and the continuity of $\eta$, for $\ep$ small enough we have $\eta(M_\ep)\geq \eta(M)/2$ so that~$[M_\ep]_0\lesssim[M]_0$ and we thus can infer from Lemma~\ref{lem:goalL2:weakbesov} that for any $V\in {\mathbb E}^\frac dp_T$ such that~$L_M V\in {\mathbb Y}^\frac dp_T$ and $\langle V(t)\rangle = 0$ for all $t\in[0,T]$,
 $$
 \|V\|_{{\mathbb E}^\frac dp_T}  \lesssim_{T,\omega_M,[M]_0}  \|V(0)\|_{B^\frac dp_{p,1}} + \| L_{M_\ep} V \|_{{\mathbb Y}^\frac dp_T }  + \Big(1+\|M_\ep\|^2_{ L^\infty ([0,T];B^{\frac dp+2}_{p,1})  }\Big)  \|V\|_ {{\mathbb Y}^\frac dp_T}  \, .
 $$
Now since
\begin{align*}
L_{M_\ep }V-L_{M }V = \sum_{k=1}^d \partial_k \big[(M_\ep-M)\partial_k V\big] \, ,
\end{align*}
we have thanks to Proposition~\ref{prop:rhsBesov} 
$$
\|(L_{M_\ep}-L_M)V\|_{ {\mathbb Y}^\frac dp_T} \lesssim \|M_\ep-M\|_{L^\infty } \|V\|_{ {\mathbb Y}^{\frac dp+2}_T} +
\int_0^T \|  (M_\ep-M)(t)\|_{B^{\frac dp+2}_{p,1}} \|V(t)\|_{B^{\frac dp}_{p,1}} \, \dd t\, .
$$
so the estimate above becomes thanks to~(\ref{ineq:Mep2besov}), recalling that the multiplicative constant behind~$ \lesssim_{T,[M]_0}$ is an increasing function~$  g(T+\omega_M+[M]_0)$,
$$
\begin{aligned}
 \|V\|_{{\mathbb E}^\frac dp_T} & \lesssim_{T,\omega_M,[M]_0}  \|V(0)\|_{B^\frac dp_{p,1}} + \| L_{M } V \|_{{\mathbb Y}^\frac dp_T }  \\
 & + 
\int_0^T\big(  \|  M(t)\|_{B^{\frac dp+2}_{p,1}}+ \|  M_\ep(t)\|_{B^{\frac dp+2}_{p,1}}\big) \|V(t)\|_{B^{\frac dp}_{p,1}} \, \dd t \\
&\quad + \Big(1+\|M_\ep\|^2_{ L^\infty ([0,T];B^{\frac dp+2}_{p,1})  }\Big)  \|V\|_ {{\mathbb Y}^\frac dp_T}   \, .
\end{aligned}
$$
The result is proved thanks to~(\ref{ineq:Mep11besov}),~(\ref{ineq:Mep3besov}) and Gronwall's lemma, provided $\ep$  can be replaced by some decreasing function of $T+\omega_M+[M]_0$:  that was done in the Sobolev case and the proof is identical here. \qed

\subsubsection{Proof of Lemma~\ref{lem:goalL2:weakbesov}}\label{subsec:proof_lemmabesov} 
  The idea, as in the Sobolev case,  is to  reduce to the case of a constant in space matrix field   by a partition of unity of~$   \T^d$. 

\medskip

We start with a localization lemma.

 \begin{lem}\label{lem:locbesov}
Fix~$\gamma \in (0,1)$,  $M\in\mathscr{C}^{0}(Q_T;\mathscr{P})\cap \L^\infty ([0,T];B^{\frac dp+2}_{p,1}) $. 
For~$\ep$ small enough so that~$\omega_M(\ep)$ is smaller than a decreasing function of~$T+[M]_0  $, the following holds. For any $V\in {\mathbb E}^\frac dp_T$ such that~$L_{M }V $ belongs to~$ {\mathbb Y}^\frac dp_T$ and any smooth bump function $\theta$ supported in a ball of~$\T^d$ of radius~$\ep$, there holds  
 \begin{multline*}
    \|\mathbb{P}_0(\theta V)\|_{{\mathbb E}^\frac dp_T}  \lesssim_{T,[M]_0} \big(1+ \|\theta\|^2_{B^{\frac dp+2}_{p,1}} \big)\Big(\|\mathbb{P}_0 V (0)\|_{B^{\frac dp }_{p,1}}  +  \|    L_{M }V\|_{{\mathbb Y}^\frac dp_T}  \\+ \gamma \|V\|_{{{\mathbb Y}^{\frac dp +2}_T }}  +\frac  1 \gamma\big(1+  \|M\|^2_{{L^\infty_T B^{\frac dp+2}_{p,1} }} \big)  \|V\|_{{{\mathbb Y}^{\frac dp }_T }} \Big) \,  .
  \end{multline*}
      \end{lem}
      \begin{proof}
As in the proof of Lemma~\ref{lem:loc}, we consider  $M^\star : t\mapsto M(t,x^\star)$ for which we have for all $t\in[0,T]$
 \begin{equation}\label{PstarP:bes} 
\forall  x \in \mbox{Supp}\,  \theta \, , \quad \vertiii{M(t,x)-M^\star(t)} \leq \omega(\ep)\,, 
 \end{equation}
and we write
  \begin{align} \label{eq:defF:bes}
L_{M^\star}(\theta V) := \partial_t( \theta V)-  M^\star \Delta (\theta V)= \theta  L_{M}V +  \sum_{k=1}^d \big( \partial_kR_k^\star +      S^\star_k  \big)\,.
    \end{align}
   Using Proposition~\ref{prop:systlinbesov}, we get
          \begin{align*}
       \|\mathbb{P}_0(\theta V)\|_{{\mathbb E}^\frac dp_T}  \lesssim_{T,[M]_0}   \|\mathbb{P}_0(\theta V)(0)\|_{B^{\frac dp }_{p,1}}  +  \|\mathbb{P}_0 L_{M^\star}(\theta V)\|_{{\mathbb E}^\frac dp_T} \,.
\end{align*}
  Using \eqref{eq:defF:bes} we  thus have
     \begin{equation}\label{estimateP0thetaV:bes}
     \begin{aligned}
    \|\mathbb{P}_0(\theta V)\|_{{\mathbb E}^\frac dp_T} & \lesssim_{T,[M]_0}   \|\mathbb{P}_0(\theta V)(0)\|_{B^{\frac dp }_{p,1}}  +  \|\mathbb{P}_0  (\theta  L_{M}V) \|_{{\mathbb E}^\frac dp_T}  \\
    &\quad   + \sum_{k=1}^d\Big( \|\partial_k R_k^\star\|_{{\mathbb E}^\frac dp_T}  + \|\mathbb{P}_0 S_k^\star\|_{{\mathbb E}^\frac dp_T} \Big)\,.
  \end{aligned}
  \end{equation}
But we have, using \eqref{PstarP:bes} along with Proposition~\ref{prop:rhsBesov},
$$
 \begin{aligned}
\|\partial_k R_k^\star\|_{{\mathbb Y}^\frac dp_T}& \leq \| R_k^\star\|_{{\mathbb Y}^{\frac dp+1}_T} \\
& \lesssim \|M - M^\star\|_{L^\infty } \|\nabla (\theta V)\|_{ {\mathbb Y}^{\frac dp+1}_T}+
\int_0^T \|  (M - M^\star)(t)\|_{B^{\frac dp+2}_{p,1}} \|\theta V(t)\|_{B^{\frac dp}_{p,1}} \, \dd t\\
&\lesssim  \omega(\ep) \| \nabla(\theta V)  \|_{{\mathbb Y}^{\frac dp+1}_T} +   \|\theta \|_{ B^\frac dp_{p,1}} \int_0^T 
  \|  M(t)\|_{ B^{\frac dp+2}_{p,1}} \|V(t)\|_{ B^{\frac dp}_{p,1}}\, \dd t \,.
 \end{aligned}
 $$
 In particular,   if ~$\ep $  is small enough (depending only on~$T+[M]_0$),  estimate \eqref{estimateP0thetaV:bes} becomes
    \begin{equation}\label{estimateP0thetaVbisbesov}
 \begin{aligned}
  \|\mathbb{P}_0(\theta V)\|_{{\mathbb E}^\frac dp_T}   \lesssim_{T,[M]_0}  \|\mathbb{P}_0(\theta V)(0)\|_{B^{\frac dp }_{p,1}}
  +  \| \mathbb{P}_0(\theta  L_{M }V)\|_{{\mathbb Y}^\frac dp_T}  \\ + \int_0^T 
  \|  M(t)\|_{ B^{\frac dp+2}_{p,1}} \|V(t)\|_{ B^{\frac dp}_{p,1}}\, \dd t+ \sum_{k=1}^d  \|\mathbb{P}_0\,S_k^\star\|_{{\mathbb Y}^\frac dp_T} \, .
   \end{aligned}
\end{equation}
 Now let us estimate   the other terms on   the right-hand side of~(\ref{estimateP0thetaVbisbesov}). First, we notice that since~$B^\frac dp_{p,1}$ is an algebra, then 
  $$
 \begin{aligned}
 \|  \mathbb{P}_0(\theta   L_{M }V)\| _{{\mathbb Y}^{\frac dp}_T} & \leq \|   \theta   L_{M }V\| _{{\mathbb Y}^{\frac dp}_T}  \\
 &\lesssim  \|\theta \|_{ B^\frac dp_{p,1}}    \|   L_{M }V(t) \| _{{\mathbb Y}^{\frac dp}_T}\, .
  \end{aligned}
 $$
   For the first term involving~$S^\star_k$ we   write    
 $$
 \begin{aligned}
  \big \|   \mathbb{P}_0(     M  (\partial_k V(t))( \partial_k\theta ) ) \big\|_{ {\mathbb Y}^{\frac dp}_T } & \lesssim  \|\nabla \theta\|_{ B^\frac dp_{p,1} } \int_0^T  \| \partial_k V(t)\|_{B^\frac dp_{p,1}} 
  \|M(t)\| _{  B^\frac dp_{p,1}}  \, \dd t  \\
  & \lesssim \|\nabla \theta\|_{ B^\frac dp_{p,1} }  \int_0^T  \|   V(t)\|_{B^\frac dp_{p,1}} ^\frac12 \|  V(t)\|_{B^{\frac dp+2}_{p,1}} ^\frac12
  \|M(t)\| _{  B^\frac dp_{p,1}}  \, \dd t   \end{aligned}$$
by interpolation, whence for any~$
\gamma>0$
 $$   \big \|   \mathbb{P}_0(     M  (\partial_k V(t))( \partial_k\theta ) ) \big\|_{ {\mathbb Y}^{\frac dp}_T }\leq \gamma  \|  V\|_{{\mathbb Y}^{\frac dp+2}_T} + \frac C\gamma \|  \theta\|^2_{ B^{\frac dp+1}_{p,1} }  \int_0^T  \|   V(t)\|_{B^\frac dp_{p,1}}\|M(t)\|^2 _{  B^\frac dp_{p,1}}   \, \dd t \, .
 $$
    The two other terms are simply estimated by
  $$
   \begin{aligned}
 \big \|   \mathbb{P}_0(    MV \partial_k^2 \theta ) \big\|_{{\mathbb Y}^{\frac dp}_T }   & \lesssim  \|\Delta \theta\|_{ B^\frac dp_{p,1} } \int_0^T  \|  V(t)\|_{B^\frac dp_{p,1}} 
  \|M(t)\| _{  B^\frac dp_{p,1}}  \, \dd t  \\
  & \lesssim  \|  \theta\|_{ B^{\frac dp+2}_{p,1} }  \int_0^T  \|  V(t)\|_{B^\frac dp_{p,1}} 
  \|M(t)\| _{  B^\frac dp_{p,1}}  \, \dd t
  \end{aligned} $$
  and 
  $$
   \begin{aligned}
  \big \|   \mathbb{P}_0(    ( \partial_k M ) V( \partial_k\theta)    ) \big\|_{{\mathbb Y}^{\frac dp}_T } & \lesssim \|\nabla \theta\|_{ B^\frac dp_{p,1} } \int_0^T  \|  V(t)\|_{B^\frac dp_{p,1}} 
  \|\nabla M(t)\| _{  B^\frac dp_{p,1}}   \, \dd t \\
  & \lesssim  \|  \theta\|_{ B^{\frac dp+1}_{p,1} }  \int_0^T  \|  V(t)\|_{B^\frac dp_{p,1}} 
  \|  M(t)\| _{  B^{\frac dp+1}_{p,1}}   \, \dd t\, .
   \end{aligned} $$
   Finally, going back to \eqref{estimateP0thetaVbisbesov}, the proof of Lemma~\ref{lem:locbesov} is over.     \end{proof}

We now proceed to the proof of Lemma~\ref{lem:goalL2:weakbesov}, using the notation of the proof of Lemma~\ref{lem:goalL2:weak}.    We apply the result of Lemma~\ref{lem:locbesov} with~$\theta =   \theta^{j}_\ep  $, for ~$1 \leq j \leq J_M$. We have thus
 \begin{multline*}
    \|\mathbb{P}_0( \theta^{j}_\ep  V)\|_{{\mathbb E}^\frac dp_T}  \lesssim_{T,[M]_0} \big(1+ \| \theta^{j}_\ep \|^2_{B^{\frac dp+2}_{p,1}} \big)\Big(\| V (0)\|_{B^{\frac dp }_{p,1}}  +  \|    L_{M }V\|_{{\mathbb Y}^\frac dp_T}  \\+ \gamma \|V\|_{{{\mathbb Y}^{\frac dp+2}_T }}  +\frac 1 \gamma\big(1+  \|M\|^2_{{L^\infty_T B^{\frac dp+2}_{p,1} }} \big)  \|   V  \|_{{\mathbb Y}^\frac dp_T}  \Big) \,  .
  \end{multline*}
  Note that the constant behind $\lesssim_{T,[M]_0}$ is  increasing with~$T+[M]_0 $ and does not depend on~$j$. Using  that $\omega_M(\ep)$ can be chosen as a decreasing  function of~$T+[M]_0 $ we find  
   \begin{multline*}
    \|\mathbb{P}_0( \theta^{j}_\ep  V)\|_{{\mathbb E}^\frac dp_T}  \lesssim_{T,\omega_M,[M]_0}  \| V (0)\|_{B^{\frac dp }_{p,1}}  +  \|    L_{M }V\|_{{\mathbb Y}^\frac dp_T}  \\+ \gamma \|V\|_{{{\mathbb Y}^{\frac dp+2}_T }}  + \frac1\gamma\big(1+  \|M\|^2_{{L^\infty_T B^{\frac dp+2}_{p,1} }} \big) \|   V  \|_{{\mathbb Y}^\frac dp_T}   \,  .
 \end{multline*}
Then we use the fact that~$V = \mathbb{P}_0 V$ and that the family~$(\theta^{j}_\ep)_{1 \leq j \leq J_M}$ is a partition of unity, so
$$
\|V\|_{{\mathbb E}^\frac dp_T} = \Big \| \sum_{1 \leq j \leq J_M}  \mathbb{P}_0 (\theta^{j}_\ep  V)\Big\| _{{\mathbb E}^\frac dp_T}  \leq  \sum_{1 \leq j \leq J_M} \|  \mathbb{P}_0 (\theta^{j}_\ep  V)\| _{{\mathbb E}^\frac dp_T}
$$
and thus
  \begin{multline*}
  \|V\|_{{\mathbb E}^\frac dp_T}   \lesssim_{T,\omega_M,[M]_0}  J_M\Big(\| V (0)\|_{B^{\frac dp }_{p,1}}  +  \|    L_{M }V\|_{{\mathbb Y}^\frac dp_T}  \\+ \gamma \|V\|_{{{\mathbb Y}^{\frac dp+2}_T }}  + \frac1\gamma\big(1+  \|M\|^2_{{L^\infty_T B^{\frac dp+2}_{p,1} }} \big)  \|   V  \|_{{\mathbb Y}^\frac dp_T} \Big)
 \end{multline*}
 and we conclude by choosing~$\gamma$ small enough (depending only on~$J_M, T + [M]_0$), recalling that~$J_M \sim \ep^{-d}$. 
Lemma~\ref{lem:goalL2:weakbesov} is proved. \qed

   \subsection{Conclusion}
   To conclude the proof we shall use the linear estimate provided by Theorem~\ref{thm:PetrovskiiBesov} to implement a fixed point argument.   Let us set 
   $$
 \mathbb G^p_T:={\mathbb E}^\frac dp_T \cap \mathcal C^0(Q_T)\, .
   $$
       Given $(U^0,F)\in B^\frac dp_{p,1}\cap  {\mathbb Y}^\frac dp_T$, we consider the following map
\begin{align*}
  \Theta :  \mathbb G^p_T&\longrightarrow {\mathbb E}^\frac dp_T  \\
   U &\longmapsto U^\star,
\end{align*}
where $U^\star$ is the only element   of~${\mathbb E}^\frac dp_T$ solving $L_{A(U)} U^\star = F$ with~$U(0) = U^0$. 
 
 Before starting the fixed point procedure, let us check that~$\Theta$ maps~$ \mathbb G^p_T$ onto itself. We recall that~$U^\star:=\Theta(U)$ solves
 $$
 \partial_tU^\star = \sum_k \partial_k \big(A(U)\partial_kU^\star\big) + F
 $$
 and we know that~$F$ belongs to~$  {\mathbb Y}^\frac dp_T$. Let us prove that the same information holds for~$\partial_k \big(A(U)\big)\partial_kU^\star\big)$. We write
 $$
 \partial_k \big(A(U)\partial_kU^\star\big) = A(U)\partial^2_kU^\star + A'(U)\partial_kU \partial_kU^\star
 $$
 and we know that smooth functions are continuous over~$B^\frac dp_{p,1}$, so since~$B^\frac dp_{p,1}$ is an algebra, it follows that
 $$
 \big\| A(U)\partial^2_kU^\star \big\| _{  {\mathbb Y}^\frac dp_T} \lesssim \big\| A(U) \big\| _{L^\infty_T(B^\frac dp_{p,1})} \|\partial^2_k U^\star\| _{  {\mathbb Y}^\frac dp_T} \lesssim \Phi( \| U  \| _{L^\infty_T(B^\frac dp_{p,1})} )\|U^\star\| _{  {\mathbb Y}^{\frac dp+2}_T} 
 $$
 with~$\Phi$ smooth and increasing. Similarly
 $$
  \big\|  A'(U)\partial_kU \partial_kU^\star
 \big\| _{  {\mathbb Y}^\frac dp_T} \lesssim  \big\| A'(U) \big\| _{L^\infty_T(B^\frac dp_{p,1})} \|\partial_k U \| _{L^2_T(B^{\frac dp+1}_{p,1})} \|\partial_k U^\star\| _{L^2_T(B^{\frac dp+1}_{p,1})} \, .
 $$
 We thus find that~$ \partial_tU^\star$ belongs to~$\L^1(0,T;B^\frac dp_{p,1})$, which implies the expected result since~$B^\frac dp_{p,1}$ is embedded in the space of continuous functions.

                Now let~$U_F:= \Theta(0)$ and fix~$r \in (0,1]$.   We start by checking that for~$r$ small enough (depending only on~$T$ and~$\|U_F\|_{{\mathbb E}^\frac dp_T}$) the ball in~${\mathbb E}^\frac dp_T$ centered at~$U_F$ and of radius~$r$  is stabilized by~$\Theta$.
                 
                Consider~$U_1 \in {\mathbb E}^\frac dp_T$   such that~$\|U_1-U_F\|_{ {\mathbb E}^\frac dp_T} \leq r$ and set~$V:=U_1^\star - U_F$.      The same argument as in the Sobolev space allows to replace the continuity constant~$[A(U_F)]_0$ by~$\|U_F\|_{\mathbb G^p_T}$ so that
                   $$
                        \|V\|_{{\mathbb E}^\frac dp_T}   \lesssim_{T,\|U_F\|_{\mathbb G^p_T},\omega_{U_F} }  \| V (0)\|_{B^{\frac dp }_{p,1}} + \|L_{A(U_F)}V\|_{\mathbb Y_T^\frac dp}
                   $$
                   and since
                   $$
                        L_{A(U_F)} V =   L_{A(U_1)} V + \sum_k \partial_k \big(\big(A(U_1) - A(U_F)\big)\partial_k V\big)
                   $$
                   we find thanks to Proposition~\ref{prop:rhsBesov} that
                   $$
           \begin{aligned}
                  \|V\|_{{\mathbb E}^\frac dp_T}  & \lesssim_{T,\|U_F\|_{\mathbb G^p_T},\omega_{U_F}}    \|L_{A(U_1)}V\|_{\mathbb E_T^\frac dp}  +  \|A(U_1)- A(U_F)\|_{\L^\infty(Q_T)}   \|V\|_{\mathbb Y_T^{\frac dp+2}} \\
                  &\quad+ \int_0^T
     \big \|(A(U_1)- A(U_F)(t)\big)\|_{B_{p,1}^{\frac dp+2}}             \| V (t)\|_{B^{\frac dp }_{p,1}}   \, \dd t
                  \, .   
    \end{aligned}
                   $$
                   It follows  that if~$r$ is small enough (depending only on~$T$ and~$  U_F  $) then
                   $$
                       \|V\|_{{\mathbb E}^\frac dp_T}   \lesssim_{T,  \|U_F\|_{\mathbb G^p_T} ,\omega_{U_F} }    \|L_{A(U_1)}V\|_{\mathbb Y_T^\frac dp}  + \int_0^T
     \big \|(A(U_1)- A(U_F)(t)\big)\|_{B_{p,1}^{\frac dp+2}}             \| V (t)\|_{B^{\frac dp }_{p,1}}   \, \dd t
                  \, .   
                   $$
                                           But                   $$
                  L_{A(U_1)} V = \sum_k \partial_k \big( \big(A(U_1) - A(0)\big)\partial_kU_F \big)
                   $$                   
so again
  $$
           \begin{aligned}
                  \|V\|_{{\mathbb E}^\frac dp_T}  & \lesssim_{T, \|U_F\|_{\mathbb G^p_T} ,\omega_{U_F}}      \|A(U_1)- A(0)\|_{\L^\infty(Q_T)}   \|U_F\|_{\mathbb Y_T^{\frac dp+2}} \\
                  &\quad+ \int_0^T
     \big \|(A(U_1)- A(0)(t)\big)\|_{B_{p,1}^{\frac dp+1}}             \| U_F (t)\|_{B^{\frac dp+1 }_{p,1}}   \, \dd t\\
                &\quad+ \int_0^T
     \big \|(A(U_1)- A(U_F)(t)\big)\|_{B_{p,1}^{\frac dp+2}}             \| V (t)\|_{B^{\frac dp }_{p,1}}   \, \dd t
                  \, .   
    \end{aligned}
                   $$
                   It remains to use the fact that
                   $$
                \begin{aligned}
      \int_0^T
     \big \|(A(U_1)- A(0)(t)\big)\|_{B_{p,1}^{\frac dp+1}}             \| U_F (t)\|_{B^{\frac dp+1 }_{p,1}}   \, \dd t &\lesssim \|(A(U_1)- A(0) \big)\|_{L^2_T(B_{p,1}^{\frac dp+1})}  \\
  & \qquad\qquad  \qquad\qquad \times   \| U_F\|_{L^2_T(B_{p,1}^{\frac dp+1})} \, , 
  \end{aligned}
               $$
and we conclude thanks to a Gronwall estimate: we can find~$T$ small enough so that~$     \|V\|_{{\mathbb E}^\frac dp_T} \leq r$.

To conclude we need to prove that~$\Theta$ is Lipschitz on the ball  of~${\mathbb E}^\frac dp_T$ centered at~$U_F$ and of radius~$r$ with a Lipschitz constant smaller than one (for~$r$ and~$T$ small enough). This follows the same lines as the above computations: we fix~$U_1$ and~$U_2$ such that~$\| U_i-U_F\|_{{\mathbb E}^\frac dp_T} \leq r$ and we note that thanks to Proposition~\ref{prop:rhsBesov}  
$$
     \begin{aligned}
          \|U_1^\star-U_2^\star\|_{{\mathbb E}^\frac dp_T} &  \lesssim_{T, \|U_1\|_{\mathbb G^p_T} ,\omega_{U_1} }  \int_0^T   \big \|(A(U_1)- A(U_2)(t)\big)\|_{B_{p,1}^{\frac dp+1}}    \|U_2^\star(t)\|_{B_{p,1}^{\frac dp+1}}  \, \dd t\\
          &\quad +  \int_0^T   \big \|(A(U_1)- A(U_2)(t)\big)\|_{\L^\infty}    \|U_2^\star(t)\|_{B_{p,1}^{\frac dp+2}}   \, \dd t \\
    &    \lesssim_{T, r,\|U_F\|_{\mathbb G^p_T} ,\omega_{U_F}  }         \|U_1 -U_2 \|_{{\mathbb E}^\frac dp_T} (r+ \|U_F\|_{L^1_TB_{p,1}^{\frac dp+2}} ) \\
    &\quad + 
       \|U_1 -U_2 \|_{{\mathbb E}^\frac dp_T}(r+ \|U_F\|_{L^2_TB_{p,1}^{\frac dp+1}} ) 
\end{aligned}
   $$
   where we have used that since~$\|U_1-U_F\|_{\mathbb{E}_T^{d/p}} \leq r$ then in particular for any~$(t,x) $ and~$ (t',x')  $ in~$Q_T$ there holds
   $$
   |U_1(t,x) - U_1(t',x')| \leq 2r + |U_F(t,x)-U_F(t',x')|  
   $$
  whence  $$
   \omega_{U_1} \leq r + \omega_{U_F} \, .
   $$
The result follows choosing~$r$ and~$T$ small enough.
     \qed

    \appendix

\section{Sobolev estimates}\label{sec:sobo}
 Let us start by stating this very classical lemma, the proof of which is recalled for the convenience of the reader.
\begin{lem}\label{lem:sobo1}
Fix $s\in\R$. If $f\in \L^\infty(0,T;\H^s(\T^d)) \cap  Y_T^{s+1}$ with $\partial_t f \in Y_T^{s-1}$, then $f\in X_T^s$. 
\end{lem}
\begin{proof}
Recalling the definition of $X_T^s$, we only need to prove continuity in time with values in $\H^s(\T^d)$.  If $(f_n)_n$ is a sequence of smooth functions approaching $f$ in $Y_T^{s+1}$ such that~$(\partial_t f_n)_n$ converges to $\partial_t f$ in $Y_T^{s-1}$ and $(f_n)_n$  is uniformly bounded in $\L^\infty(0,T;\H^s(\T^d))$, a direct computation gives for $n,k\in\N$ and $t,r\in[0,T]$
  \begin{multline*}
\|f_n(t)-f_k(t)\|_{\H^s(\T^d)}^2 = \|f_n(r)-f_k(r)\|_{\H^s(\T^d)}^2 \\+2 \int_r^t \big\langle (f_n-f_k)(t'),(\partial_t f_n-\partial_t f_k)(t') \big\rangle_{\H^s(\T^d)}\,\dd t' \, , 
  \end{multline*}
 where~$\langle \cdot,\cdot \rangle_{\H^s(\T^d)}$ denotes the scalar product in~$\H^s(\T^d)$. Expressing this dot product in terms of Fourier coefficients and using Cauchy-Schwarz's inequality, we infer for any $\varphi,\psi\in\mathscr{C}^\infty(\T^d)$ that $|\langle f,g\rangle_{\H^s(\T^d)}|\leq \|f\|_{\H^{s+1}(\T^d)}\|g\|_{\H^{s-1}(\T^d)}$. We have therefore, using another time Cauchy-Schwarz's inequality (for the time integral)
    \begin{align*}
\|f_n(t)-f_k(t)\|_{\H^s(\T^d)}^2 \leq \|f_n(r)-f_k(r)\|_{\H^s(\T^d)}^2 + 2\|f_n-f_k\|_{Y_T^{s+1}}\|\partial_t f_n -\partial_t f_k\|_{Y_T^{s-1}} \, .
    \end{align*}
    Integrating in $r\in[0,T]$ we get
    \begin{align*}
T\|f_n(t)-f_k(t)\|_{\H^s(\T^d)}^2 \leq \|f_n-f_k\|_{Y_T^s}^2 + 2T \|f_n-f_k\|_{Y_T^{s+1}}\|\partial_t f_n -\partial_t f_k\|_{Y_T^{s-1}} \,  ,
    \end{align*}
    from which we infer that $(f_n)_n$ is a Cauchy sequence in $\mathscr{C}^0([0,T];\H^s(\T^d))$, because it is the case for $(f_n)_n$ in $Y_T^{s+1}$ and $(\partial_t f_n)_n$ in $Y_T^{s-1}$. This entails the announced regularity for $f$. $\qedhere$ 
\end{proof}

\begin{lem}\label{lem:sobo2}
  Fix $s>d/2$. There exists $\alpha_s\in(0,1)$ such that any $f\in X_T^s$ satisfying that~$\partial_t f \in Y_T^{s-1}$ actually belongs to $\mathscr{C}^{0,\alpha_s}(Q_T)$ with an estimate
  \[\|f\|_{\mathscr{C}^{0,\alpha_s}(Q_T)}\lesssim_s \|f\|_{X_T^s} + \|\partial_t f\|_{Y_T^{s-1}} \, .\]
\end{lem}
\begin{proof}
  Using Cauchy-Schwarz's inequality and the assumption on the time derivative, we have
  \[\|f(t_1)-f(t_2)\|_{\H^{s-1}(\T^d)} \leq \int_{t_2}^{t_1}\| \partial_t f(r)\|_{\H^{s-1}(\T^d)}\,\dd r \leq |t_1-t_2|^{1/2}\|\partial_t f\|_{Y_T^{s-1}},\]
which establishes that~$f$ belongs to the space~$\mathscr{C}^{0,1/2}([0,T];\H^{s-1}(\T^d))$. Now, we choose~$\sigma $ in~$(d/2,s)$ such that $\sigma>s-1$ so that by interpolation $$\H^\sigma(\T^d) = [\H^{s-1}(\T^d),\H^{s}(\T^d)]_\theta\, .$$  We have thus for $t_1\neq t_2\in [0,T]$  \begin{align*}\|f(t_1)-f(t_2)\|_{\H^\sigma(\T^d)} &\leq \|f(t_1)-f(t_2)\|_{\H^{s-1}(\T^d)}^\theta \|f(t_1)-f(t_2)\|^{1-\theta}_{\H^{s}(\T^d)}\\
                                       &\leq |t_1-t_2|^{\theta/2} \|\partial_t f\|_{\H^{s-1}(\T^d)}^{\theta} 2^{1-\theta} \|f\|_{X_T^s}^{1-\theta} \,  ,\end{align*}
                                     so that
                                     \begin{align*}
\frac{\|f(t_1)-f(t_2)\|_{\H^{\sigma}(\T^d)}}{|t_1-t_2|^{\theta/2}} \lesssim \|\partial_t f\|_{\H^{s-1}(\T^d)} + \|f\|_{X_T^s} \,  ,
                                     \end{align*}
                                     and the conclusion follows using the Sobolev embedding $\H^\sigma(\T^d)\hookrightarrow \mathscr{C}^{0,\beta}(\T^d)$ which holds for some $\beta\in(0,1)$. $\qedhere$ \end{proof}
\begin{lem}\label{lem:lipsob}
For  $\sigma>d/2$ and $\Phi$ a smooth function, there exists an increasing function $\ffi$ for which, for any elements $f,g\in\H^{\sigma}(\T^d)$ 
\begin{equation}\label{ineq:Alip}
\begin{aligned}
  \|\Phi(f)-\Phi(g)\|_{\H^\sigma(\T^d)} & \leq \ffi(\|f\|_{\infty}+\|g\|_\infty) \\
 & \quad  \times \textcolor{black}{\Big[ \|f-g\|_{\H^\sigma(\T^d)}+(\|f\|_{\H^\sigma(\T^d)}+\|g\|_{\H^\sigma(\T^d)})\|f-g\|_\infty}\Big]\, .
                                     \end{aligned}
\end{equation}\end{lem}
\begin{proof}
   See for instance \cite[Corollary 2.91]{BCD}.
\end{proof}
\section{Littlewood-Paley theory}\label{sec:LP}
\subsection{Definitions}
In this section we present the elements of Littlewood-Paley theory that are used in this study. We recall  (see for instance \cite{BCD} where the construction is carried out in~$\R^d$ but is easily adapted to the periodic case) that the basic idea is  to consider a dyadic partition of unity in~$\R^d$
$$
1 = \widehat\chi +  \sum_{j \geq 0} \widehat \varphi (2^{-j} \cdot)
$$
 where~$\widehat\chi$ and~$\widehat\varphi$ (the Fourier transforms of two smooth functions~$\chi$ and~$\varphi$) are smooth, radial functions, taking values in~$[0,1]$ and supported respectively in the ball~$B(0,4/3)$ and the ring~$[3/4,8/3]$.
We set for any integer~$j \geq 0$ and any function~$f$ defined on~$\T^d$
\begin{equation}\label{eq:def Delta j}
\Delta_j f:= f \star 2^{jd} \varphi (2^j  \cdot )  
\end{equation}
and
$$
\Delta_{-1} f:= f \star  \chi 
\, .
$$
Note that in particular
$$
 \sum_{j \geq -1}\Delta_j  = \mbox{Id} \, .
$$
 Finally for~$j < -1$ we set~$\Delta_{ j} = 0$ and
 $$
\forall j \geq 0 \, , \quad S_j:= \sum_{j'=-1}^{j-1}\Delta_{j'} \, .
  $$
 Writing these formulas in Fourier space we see that the support of the Fourier transform of~$ \Delta_j f$ lies in a ring of size~$2^j$
 if~$j \geq 0$ and in the unit ball if~$j = -1$ (this corresponds therefore to the average of~$f$). Moreover the functions~$\widehat\chi$ and~$\widehat\varphi$ are designed so as to have
 $$
 |j-j' | \geq 2 \Longrightarrow \mbox{Supp}\, \widehat \varphi (2^{-j} \cdot) \cap \mbox{Supp} \, \widehat \varphi (2^{-j'} \cdot)  = \emptyset
 $$
 and
  $$
j \geq 1 \Longrightarrow \mbox{Supp}\,  \widehat \chi \cap \mbox{Supp}\, \widehat  \varphi (2^{-j} \cdot)  = \emptyset\, .
 $$
 The following Bernstein inequality is used many times in this paper
 \begin{equation}\label{eq:bernstein}
\forall \alpha \in \N^d \, , \quad \forall 1 \leq p \leq q \leq \infty \,, \quad \big\|\partial^\alpha \Delta_j  f \big\|_{q} \lesssim 2^{j \big(|\alpha| + d (\frac1p-\frac1q)\big)}   \big\|\Delta_j   f\big\|_{p}\, .
\end{equation}
For the convenience of the reader let us recall how to prove this inequality: we consider a smooth, compactly supported function~$\widehat{\widetilde \varphi}$ such that~$
\widehat{\widetilde \varphi }\widehat \varphi \equiv 1$ and we note that
$$
\Delta_j f = \Delta_j f \star 2^{jd} \widetilde\varphi (2^j  \cdot )   \, .
$$
Then we write
$$
\partial^\alpha \Delta_jf = \Delta_j f \star 2^{jd+j|\alpha|} (\partial^\alpha\widetilde\varphi) (2^j  \cdot ) 
$$
and we conclude by Young's inequality
$$
 \big\|\partial^\alpha \Delta_j  f \big\|_{q} \leq  2^{jd+j|\alpha|}\big\|\partial^\alpha \Delta_j  f \big\|_{p} 
  \big\| \partial^\alpha\widetilde\varphi (2^j  \cdot )  \big\|_{r} \, , \quad 1+\frac1q =  \frac1p+ \frac1r
$$
and the result~\eqref{eq:bernstein}
 follows.
It is also useful to note (see for instance~\cite[Lemma 2.4]{BCD}) that there is a constant~$c>0$ such that
 \begin{equation}\label{eq:heat}
 \forall j \geq 0 \, , \quad  \forall 1 \leq p  \leq \infty \,, \quad \big\|e^{t\Delta} \Delta_j  f \big\|_{p} \lesssim  e^{-ct 2^{2j}}  \big\|\Delta_j   f\big\|_{p}\, .
\end{equation}
\medskip

With this construction,  Sobolev spaces can be defined by  the equivalent norm
  $$
\|f\|_{\H^s }\sim \Big\| 2^{js} \|   \Delta_j f\|_{\L^2(\T^d)} \Big\|_{\ell^2(\mathbb Z)} \, ,
$$
and the Besov spaces dealt with in this paper are given, for    any~$p \in [1,\infty)$  and~$s \in \R$ by the norm
 $$
 \|f\|_{B^s_{p,1}}:= \sum_{j \geq -1} 2^{js} \|\Delta_j f\|_{\L^p(\T^d)} \, .
 $$
 It is well-known (see~\cite{BCD}) that any function  in~$B^\frac dp_{p,1}$ is   continuous.

One major interest of this theory is the paraproduct algorithm due to Bony \cite{bony}  :  decomposing formally any two tempered distributions~$f$ and~$g$ as
$$
f = \sum_{j  } \Delta_j f \quad \mbox{and} \quad g = \sum_{j } \Delta_j g
$$
then the product~$fg$ can formally be decomposed into three parts
$$
fg= T_fg+T_gf+R(f,g) \,, \quad T_fg:=\sum_{j }S_{j-1} f \,\Delta_j g \,, \quad R(f,g) :=\sum_{|j-j'|\leq 1 }  \Delta_j f \,\Delta_{j'}g \, .
$$
On the Fourier side, thanks to the support properties of~$\varphi$ and~$\chi$, each term~$S_{j-1} f \Delta_j g$ of the paraproduct~$ T_fg$ is supported in a ring of size~$2^j$ (hence the sum is well defined under mild assumptions on~$f$ and~$g$: for instance~$f$ bounded and~$g$ in any Sobolev or H\"older space). The remainder term~$ R(f,g)$ however is not always well defined. On the Fourier side, each term~$ \Delta_j f \Delta_{j'}g$ is supported in a ball of size~$2^j$ (since~$j \sim j'$) and the sum only makes sense if the regularities of~$f$ and~$g$ sum up to a positive number. We refer to~\cite{BCD} for instance for more on this.

In this paper we sometimes use a less sharp decomposition, writing  formally
$$
fg= \sum_{j }S_{j-1} f \,\Delta_j g +  \sum_{j }\Delta_j  f \, S_{j+2}g \, .
$$

  \subsection{A product law}
  Let us start by proving the following useful estimate.
   \begin{prop}\label{prop:rhsSobo}
 The following estimate holds, for any smooth enough~$f$ and~$g$ :
$$
\|f g\|_{\H^s(\T^d)} \lesssim \|f\|_{\infty} \|g\|_{\H^s(\T^d)}+\|f\|_{\H^{s+1}(\T^d)}\|g\|_{\H^{s-1}(\T^d)}.
$$
\end{prop}
\begin{proof}
  As usual we treat adequately each term of the decomposition
  \begin{align*}
fg = T_f g + T_g f  + R(f,g)\, .
  \end{align*}
  For any natural integer $j$ there holds $\|S_{j-1} f\|_\infty\leq \|f\|_\infty$ so that $$\|T_f g\|_{\H^s(\T^d)} \leq \|f\|_\infty \|g\|_{\H^s(\T^d)}$$ follows directly. We aim to bound the two other terms by $\|f\|_{\H^{s+1}(\T^d)}\|g\|_{\H^{s-1}(\T^d)}$. First, as each term of the sum defining $T_g f$ is localized around $2^j$, we have for any natural integer $k$ using Bernstein's inequality
  \begin{align*}
    2^{ks}\|\Delta_k T_g f\|_2 &= 2^{ks} \sum_{j\sim k} \|S_{j-1} g \Delta_j f\|_2\\ &\lesssim 2^{ks} \sum_{j'\lesssim j \sim k} 2^{j'\frac{d}{2}}\|\Delta_{j'} g\|_2\|\Delta_j f\|_2\\
    &= 2^{ks} \sum_{j'\lesssim j \sim k} 2^{j'(\frac{d}{2}-(s-1))} \stackrel{:=g_{j'}}{\overbrace{2^{j'(s-1)}\|\Delta_{j'} g\|_2}}2^{-j(s+1)}\stackrel{:=f_j}{\overbrace{2^{j(s+1)}\|\Delta_{j} f\|_2}}\\
    &\lesssim \sum_{j'\lesssim j\sim k} 2^{-(j-j')} g_{j'} f_j\, ,
  \end{align*}
  where we crucially used $s>d/2$, $j'\lesssim j$ and $j\sim k$ for the last inequality. Since the~$\ell^2(\Z)$ norms of $(g_j)_j$ and $(f_j)_j$ are respectively $\|g\|_{\H^{s-1}(\T^d)}$ and $\|f\|_{\H^{s+1}(\T^d)}$, we have first by Cauchy-Schwarz inequality \[2^{ks}\|\Delta_k T_g f\|_2 \lesssim \sum_{j'\lesssim j\sim k} 2^{-(j-j')} g_{j'} f_j = \sum_{j\sim k} \Big(\sum_{j'\lesssim j}  2^{-(j-j')} g_{j'}\Big) f_j \lesssim \|g\|_{\H^{s-1}(\T^d)} \sum_{j\sim k} f_j\, ,\] 
and taking the $\ell^2(\Z)$ norm in $k$ we recover $$\|T_g f\|_{\H^s(\T^d)}\lesssim \|g\|_{\H^{s-1}(\T^d)}\|f\|_{\H^{s+1}(\T^d)}\, .$$
 For the remainder term $R(f,g)$ each term composing this sum has  frequencies in the ball of radius $2^j$, whence again by Bernstein's inequality (we use it here on the $f$'s blocks instead of the $g$'s)
  \begin{align*}
    2^{ks}\|\Delta_k R(f,g)\|_2 &\lesssim 2^{ks} \sum_{j'\sim j \gtrsim k} 2^{j\frac{d}{2}} \|\Delta_{j'} g\| \|\Delta_j f\|_2\\ &= 2^{ks} \sum_{j'\sim j \gtrsim k} 2^{j(\frac{d}{2}-(s+1))}2^{-j'(s-1)} g_{j'} f_j\\
    &\lesssim \sum_{j'\sim j \gtrsim k} 2^{-s(j-k)} g_{j'} f_j\,,
  \end{align*}
  where we used $s>d/2$, $j'\sim j$ and $j\gtrsim k$ in the last line. We have therefore
  \[2^{ks}\|\Delta_k R(f,g)\|_2 \lesssim \sum_{j'\sim j \gtrsim k} 2^{-s(j-k)} g_{j'} f_j = \sum_{j\gtrsim k} \stackrel{:=h_j}{\overbrace{\Big(\sum_{j'\sim j} g_{j'}\Big)}}2^{-s(j-k)}f_j\,,\]
which is nothing esle than the discrete convolution of $(h_j f_j)_j$ with $(2^{-sj}\mathbf{1}_{j\gtrsim 0})_j$, where $j\gtrsim 0$ is related to the signification of $j\gtrsim k$ and thus completely harmless here. Using Young's inequality for (discrete) convolution we infer \[\|R(f,g)\|_{\H^s(\T^d)}\lesssim \|(h_j f_j)_j\|_{\ell^1(\Z)} \leq \|(h_j)_j\|_{\ell^2(\Z)}\|(f_j)_j\|_{\ell^2(\Z)} \lesssim \|g\|_{\H^{s-1}(\T^d)}\|f\|_{\H^{s+1}(\T^d)}\,,\]
and the proof is over.
\end{proof}
A similar argument leads to the following inequality.
 \begin{prop}\label{prop:rhsBesov}
 The following estimate holds, for any smooth enough~$f$ and~$g$ :
$$
\|f \partial_k g\|_{ B^{\frac dp+1}_{p,1}} \lesssim \|f\|_{L^\infty } \|\nabla g\|_{  B^{\frac dp+	1}_{p,1}}+\min \big( 
\|  f\|_{B^{\frac dp+2}_{p,1}} \|g\|_{B^{\frac dp}_{p,1}} , \|  f\|_{B^{\frac dp+1}_{p,1}} \|g\|_{B^{\frac dp+1}_{p,1}} \big)\, .
$$
\end{prop}
\begin{proof}
We use again  the paraproduct algorithm and write
$$
f \partial_k g = T_f  \partial_k g +  T_{\partial_k g }f + R(f,\partial_k g ) \, .  
$$
On the one hand
$$
\begin{aligned}
\|T_f  \partial_k g\|_{ B^{\frac dp+1}_{p,1}} &\lesssim \sum_j \|S_{j-1}f\|_{L^\infty } \|\Delta_j  \partial_k g\|_{ L^p}2^{j (\frac dp+1)}\\
 &\lesssim\| f\|_{L^\infty }  \sum_j  \|\Delta_j  \nabla  g\|_{L^p}2^{j  (\frac dp+2)} =  \|f\|_{L^\infty} \|\nabla g\|_{  B^{\frac dp+1	}_{p,1}}\, .
\end{aligned}
$$
On the other hand
$$
\begin{aligned}
\|T_{ \partial_k g} f \|_{ B^{\frac dp+1}_{p,1}} &\lesssim \sum_{j'\lesssim j } \|\Delta_{j'} \partial_k g\|_{L^\infty }\|\Delta_j f\|_{L^p } 2^{j (\frac dp+1)}\\
 &\lesssim\sum_{j'\lesssim j }2^{j'(\frac dp+1)} \|\Delta_{j'} g\|_{ L^p} \|\Delta_j   f\|_{ L^p }2^{j (\frac dp+1)} 
 \end{aligned}
$$ which can be bounded directly by
$$
\|T_{ \partial_k g} f \|_{ B^{\frac dp+1}_{p,1}}\lesssim \|  f\|_{B^{\frac dp+1}_{p,1}} \|g\|_{B^{\frac dp+1}_{p,1}}  \, 
,
$$
or by
 $$
\begin{aligned}
\|T_{ \partial_k g} f \|_{ B^{\frac dp+1}_{p,1}} &\lesssim
  \sum_{j'\lesssim j } 2^{j'\frac dp} \|\Delta_{j'} g\|_{ L^p} 2^{j (\frac dp+2)} \|\Delta_j   f\|_{ L^p }2^{j'-j}   \\
&\lesssim\|  f\|_{B^{\frac dp+2}_{p,1}} \|g\|_{ B^{\frac dp}_{p,1}}   \, .
\end{aligned}
$$
Finally
$$
\|R(f,\partial_k g ) f \|_{ B^{\frac dp+1}_{p,1}}  \lesssim
 \sum_{j\lesssim j'} 2^{j  \frac dp }   \|\Delta_{j'} \partial_k g \|_{ L^p} \|\Delta_{j'} f\|_{  L^p }2^{j (\frac dp+1)}
$$
can be bounded by
$$
\begin{aligned}
\|R(f,\partial_k g ) f \|_{ B^{\frac dp+1}_{p,1}}   &\lesssim  \sum_{j\lesssim j'} 2^{j  (\frac {2d}p+1) }  2^{j' ( \frac dp+1) } \|\Delta_{j'} g \|_{ L^p} 2^{j' ( \frac dp+1) }\|\Delta_{j'} f\|_{L^p } 2^{-j'(  \frac {2d}p+1 )}\\
 &\lesssim\|  f\|_{B^{\frac dp+1}_{p,1}} \|g\|_{ B^{\frac dp+1}_{p,1}}   \end{aligned}
$$
or by
$$
\begin{aligned}
\|R(f,\partial_k g ) f \|_{ B^{\frac dp+1}_{p,1}}   &\lesssim  \sum_{j\lesssim j'} 2^{j  (\frac {2d}p+1) }  2^{j'  \frac dp } \|\Delta_{j'} g \|_{ L^p} 2^{j' ( \frac dp+2) }\|\Delta_{j'} f\|_{L^p } 2^{-j'(  \frac {2d}p+1 )}\\
 &\lesssim\|  f\|_{B^{\frac dp+2}_{p,1}} \|g\|_{ B^{\frac dp}_{p,1}}   \, .\end{aligned}
$$
Proposition~\ref{prop:rhsBesov} is proved. \end{proof}

\section{Petrovskii condition, Hurwitz matrices and spectral radius}\label{appendix:petrovskii}

For $B$ in $\textnormal{M}_N(\mathbb{C})$ we denote by $\textnormal{Sp}(B)$ the set of all its eigenvalues ; the spectral radius~$\rho(B)$ of $B$ is then defined by $\rho(B) = \max_{\lambda\in \textnormal{Sp}(B)} |\lambda|$. For any $\delta\in\R$ we denote by $\mathscr{P}_{\delta}$ the set of matrices $B$ for which $\textnormal{Sp}(B)\subset\{z\in\mathbb{C}\,:\,\textnormal{Re}(z)\geq \delta\}$.

\medskip

The matrix~$B$ is said to satisfy the \emph{Petrovskii condition} if $\textnormal{Sp}(B)\subset \R_{>0}+i\R$, that is if $B$ belongs to $\mathscr{P} :=\cup_{\delta>0} \mathscr{P}_\delta$. Note that in control theory and dynamical systems, the denomination \emph{Hürwitz matrix} also exists, but refers instead to a matrix whose spectrum lies in $\R_{<0}+i\R$ ; we will not use this terminology here.

\medskip

The results below, even though elementary, are of crucial importance in our analysis.

\begin{lem}\label{lem:gampet}
The map $\gamma:B\mapsto -\ln \rho(e^{-B})$ is continuous from $\textnormal{M}_N(\mathbb{C})$ to $\R$. Furthermore we have $\mathscr{P} = \gamma^{-1}(\R_{>0})$ and $B\in\mathscr{P}_{\gamma(B)}$, for any matrix $B$. 
\end{lem}
\begin{proof}
  The map $B\mapsto \textnormal{Sp}(B)$ is continuous, for the (modified) Hausdorff distance on finite sets $\dd_{\H}$ at arrival (see \cite[Theorem 5.2]{serre}). In particular, since $\rho(B) = \dd_{\H}(\textnormal{Sp}(B),\{0\})$, the spectral radius map $B\mapsto \rho(B)$ is continuous and therefore so is $\gamma$ because $\rho$ is positive on~$\textnormal{GL}_N(\mathbb{R})$. For the remaining part of the statement, the proof is ended once noticed that for any matrix $B$ one has $\rho(e^{-B}) = \max_{\lambda \in\textnormal{Sp}(B)} e^{-\textnormal{Re}(\lambda)}$, so that $\gamma(B)$ is actually the lowest real part among all eigenvalues of $B$. $\qedhere$
\end{proof}
\begin{coro}\label{coro:thetapet}
  Consider $K$ a metric compact space and $\mathscr{C}^0(K;\textnormal{M}_N(\R))$ equipped with the uniform topology. The map
  \begin{align*}
    \eta:\mathscr{C}^0(K;\textnormal{M}_N(\R)) &\longrightarrow\R\\
    M &\longmapsto \min_K \gamma \circ M 
  \end{align*}
  is well-defined and continuous. In particular, $\mathscr{C}^0(K;\mathscr{P})$ is open and lies in the set~$\cup_{\delta>0}\mathscr{C}^0(K;\mathscr{P}_\delta)$, where each $\mathscr{C}^0(K;\mathscr{P}_\delta)$ is closed. 
\end{coro}
\begin{proof}
The map is well-defined because $\gamma\circ M$ is continuous and reaches therefore its minimum on $K$. If $(M_k)_k$ converges uniformly to $M$, then so does $(\gamma \circ M_k)_k$, to $\gamma \circ M$ because $\gamma$ is continuous and $\textnormal{M}_N(\mathbb{R})$ locally compact ; this classically implies the convergence of $\eta(M_k)$ towards $\eta(M)$. By continuity if $M\in\mathscr{C}^0(K;\mathscr{P})$, then $\eta(M)>0$ and $M(K)\subset\mathscr{P}_{\eta(M)}$. In particular $\mathscr{C}^0(K;\mathscr{P}) = \eta^{-1}(\R_{>0})$  is indeed open and each $\mathscr{C}^0(K;\mathscr{P}_\delta) = \eta^{-1}(\R_{\geq \delta})$ is closed. $\qedhere$
\end{proof}
\begin{coro}\label{coro:etahold}
  Consider $K$ a metric compact space. There exists a non-increasing function $f:\R_{>0}\rightarrow\R_{>0}$ such that for any $M$ in the space~$\mathscr{C}^{0}(K;\mathscr{P})$ and any $H$ in~$\mathscr{C}^0(K;\textnormal{M}_N(\R))$ the following implication holds \[\|M-H\|_\infty < f(\|M\|_\infty+\eta(M)^{-1})\Longrightarrow \eta(H)\geq \eta(M)/2\,.\]
\end{coro}
\begin{proof}
The continuous map $\gamma$ is uniformly continuous on the $\textnormal{M}_N(\R)$ ball of radius $R_M:=1+\|M\|_\infty$. Thus, there exists $a\leq 2 R_M$ such that, for any two matrices $B_1$ and $B_2$ within this ball, the following implication holds
\[\| B_1 - B_2 \| \leq a \Longrightarrow |\gamma(B_1)-\gamma(B_2)|\leq \eta(M)/2.\]
The supremum $a_M$ of all those $a$ is a well defined real number and non-increasing in $\|M\|_\infty$ while non-decreasing in $\eta(M)$ so at the end non-increasing in $\|M\|_\infty+\eta(M)^{-1}$. To conclude the proof, we thus only have to check the following implication holds for any matrix field $H\in\mathscr{C}^0(K;\textnormal{M}_N(\R))$ 
\begin{align*}
    \| M -H\|_\infty \leq \min(1,a_M) \Longrightarrow \eta(H)\geq \eta(M)/2.
\end{align*}
Since $\|M-H\|_\infty \leq 1$, we see that $H$ takes its values in the $\textnormal{M}_N(\R)$ ball of radius $R_M$ defined above, and the previous uniform continuity can be invoked for an arbitrary $z\in K$ and two matrices $B_1=M(z)$, $B_2=H(z)$. We have therefore, since $\gamma(B_1)\geq \eta(M)$ and $\|M-H\|_\infty\leq a_M$,
\[\gamma(H(z)) \geq \eta(M)/2,\]
and this lower-bound being uniform in $z$ we have indeed $\eta(H)\geq \eta(M)/2$. $\qedhere$
\end{proof}
\section{Smooth (almost) retraction of $\R^N$ on $\R_{\geq 0}^N$}\label{app:retract}
We prove in this paragraph the following proposition. 
\begin{prop}\label{prop:quadrant}
For any open neighbourhood~$\Omega$ of~$\R_{\geq 0}^N$, there exists a smooth function~${h:\R^N\rightarrow \Omega}$ such that its restriction to~$\R_{\geq 0}^N$ is the identity map.
\end{prop}
\noindent Let us first recall the smooth Urysohn lemma.
\begin{lem}[Smooth Urysohn]\label{lem:ury}
For two disjoints and closed sets $F_0$ and $F_1$ of $\R^N$ there exists a smooth function $\ffi:\R^N\rightarrow[0,1]$ such that $\ffi^{-1}(\{1\}) = F_1$ and $\ffi^{-1}(\{0\})=F_0$.
\end{lem}
\begin{proof}
First use \cite[Theorem 2.29]{lee} to find for $k\in\{0,1\}$ smooth functions $\psi_k:\R^N\rightarrow \R$ (easily chosen non-negative) such that $\psi^{-1}_k(\{0\})=F_k$ and then letting $\ffi(x):=\psi_1/(\psi_0+\psi_1 )$ does the trick. $\qedhere$
\end{proof}
\noindent We will deduce Proposition~\ref{prop:quadrant}  from the following lemma.
\begin{lem}\label{lem:ouv}
For any open neighbourhood~$\Omega$ of~$\R_{\geq 0}^N$ there is an open set~$\widetilde \Omega$ such that ${\R_{\geq 0}^N}\subset\widetilde{\Omega}\subset\Omega$, and which is furthermore infinitely diffeomorphic to~$\R^N$.  \end{lem}
\begin{proof}
Consider~$\ffi$ the function given by  Lemma~\ref{lem:ury},  associated with the (disjoint) closed sets~$F_0:=\R^N\setminus \Omega$ and $F_1:=\R_{\geq 0}^N$. Let~$\mathbf{1}$ be the vector of~$\R^N$ with entries all equal to~$1$.  For any initial data at time~$t=0$, the differential equation
 $$\dot{\textnormal{V}} = -\ffi(\textnormal{V})\mathbf{1}$$  has a unique maximal solution, which is global since~$\ffi$  takes it values in $[0,1]$. One can 
therefore define for all~$v\in \R^N$ a smooth curve~$\textnormal{V}_v:\R\rightarrow\R^N$ equal to~$v$ at time~$t=0$ and solving that equation. Since $\ffi$ vanishes outside of~$\Omega$, for $v \in \Omega$  there holds~$\textnormal{V}_v(\R)\subset \Omega$  and flow lines passing through a point of~$\Omega$  do not exit~$\Omega$. For all $t>0$, the flow~$\Phi_t : \R^N\rightarrow\R^N$ which maps $v$ to $\textnormal{V}_v(t)$ is a  $\mathscr{C}^\infty$-diffeomorphism.  It is easy to see that $\mathring{F_1}$ is infinitely  diffeomorphic to~$\R^N$, and that is therefore also the case for~$t>0$  of the open set $ \widetilde \Omega:=\Phi_t(\mathring{F_1})$.  The set~$ \widetilde \Omega$  is contained in~$\Omega$ (since curves stemming from $\mathring{F_1}$ do not exit~$\Omega$) and finally $ \widetilde \Omega$ contains~$F_1$ since $\ffi$ is equal to~$1$ on~$F_1$. Lemma~\ref{lem:ouv} is proved.
 \qed
\end{proof}
\begin{proof}[Proof of Proposition~{\rm\ref{prop:quadrant}}]
  Using Lemma~\ref{lem:ouv}, one can  assume without loss of generality that~$\Omega $ is infinitely diffeomorphic to~$\R^N$. Thanks to that, we infer the existence of a smooth function~$\gamma:[0,1]\times  \Omega \rightarrow  \Omega$ such that $\gamma(0,\cdot)=0$ and $\gamma(1,\cdot) = \textnormal{Id}_ \Omega$. Indeed, if~$\Omega=\R^N$ then~$\gamma(t,v) := \exp\left(\frac{t-1}{t}\right)v$ does the job and the general case follows by diffeomorphism.

  \medskip
  
  Now, consider an open neighbourhood $\Omega'$ of~$\R_{\geq 0}^N$ whose closure is contained in~$\Omega$, and~$\ffi$ the function given by  Lemma~\ref{lem:ury},  associated with the (disjoint) closed sets~$F_0:=\R^N\setminus \Omega'$ and $F_1=\R_{\geq 0}^N$. The function~$h:v \mapsto   \gamma(\ffi(v),v)$  is smoothly defined  on $\Omega$.  Since it vanishes identically on~$\Omega\setminus \Omega'$,  it can be extended smoothly to~$\R^N$ by zero.  The function~$h$ thus extended takes its values in $\Omega$ and if~$v\in \R_{\geq 0}^N=F_1$, then $\ffi(v)=1$ so~$h(v)=v$. Proposition~\ref{prop:quadrant} is proved. $\qedhere$
\end{proof}

\section*{Acknowledgement}
The authors would like to thank Pierre-Louis Lions for his inspiring online lectures \cite{pilou}, and also for several fruitful discussions and his awareness concerning a previous (false) proof of Proposition~\ref{prop:systlin}. They also thank Vincent Boulard for pointing out a mistake in the previous proof of Proposition~\ref{prop:sign}. Finally they extend their gratitude to the two anonymous referees, whose suggestions and remarks greatly improved the presentation and the results of this paper.

  \printbibliography

    \end{document}